\documentclass[reqno,11pt]{amsart}

\usepackage{amssymb,graphicx,braket,amsmath,amsfonts,amsthm,enumitem}
\usepackage{hyperref}
\usepackage[utf8]{inputenc}
\usepackage{braket}
\usepackage{color}
\usepackage{amsaddr}
\usepackage{tikz}
\usetikzlibrary{shapes,arrows,er,positioning}
\usetikzlibrary{shapes.multipart}
\usepackage{stmaryrd}
\usepackage{fancyhdr}
\usepackage{array}
\usepackage{wasysym}
\usepackage{bm}
\usepackage{relsize}
\usepackage[all]{xy}
\usetikzlibrary{arrows}
\usepackage[mathscr]{euscript}
\usepackage{mathtools}
\DeclareMathAlphabet{\mathpzc}{OT1}{pzc}{m}{it}

\theoremstyle{plain}
\newtheorem{theorem}{Theorem}[section]
\theoremstyle{plain}
\newtheorem{proposition}[theorem]{Proposition}
\newtheorem*{proposition*}{Proposition}
\newtheorem*{theorem*}{Theorem}
\newtheorem{lemma}[theorem]{Lemma}
\newtheorem{assumption}[theorem]{Assumption}
\newtheorem*{assumption*}{Assumption}
\newtheorem{corollary}[theorem]{Corollary}

\theoremstyle{definition} 
\newtheorem{definition}[theorem]{Definition}

\newtheorem{remark}[theorem]{Remark}

\numberwithin{equation}{section}

 \usepackage[lmargin=1.1in,rmargin=1.1in,tmargin=1.35in,bmargin=1.35in]{geometry}
\allowdisplaybreaks

\makeatletter
\renewcommand{\section}{\@startsection
{section}
{1}
{\z@}
{-\baselineskip}
{0.8\baselineskip}
{\centering\scshape\large}} 

\renewcommand{\subsection}{\@startsection
{subsection}
{2}
{\z@}
{-0.8\baselineskip}
{0.5\baselineskip}
{\normalfont \bf \normalsize}} 

\renewcommand{\subsubsection}{\@startsection
{subsubsection}
{3}
{\z@}
{-0.8\baselineskip}
{0.5\baselineskip}
{\normalfont \bf \normalsize}} 
\makeatother

\begin{document}
\title[Quantum curves from refined topological recursion]{Quantum curves from refined topological recursion: \newline the genus 0 case}

\author{Omar Kidwai {and} Kento Osuga}

\address{Graduate School of Mathematical Sciences, The University of Tokyo\\ 3-8-1 Komaba, Meguro-ku, T\={o}ky\={o}, 153-8914, Japan}

\email{kidwai@ms.u-tokyo.ac.jp}
\email{osuga@ms.u-tokyo.ac.jp}

\begin{abstract}
We formulate geometrically (without reference to physical models) a refined topological recursion applicable to genus zero curves of degree two, inspired by Chekhov-Eynard and Marchal, introducing new degrees of freedom in the process. For such curves, we prove the fundamental properties of the recursion analogous to the unrefined case. We show the quantization of spectral curves due to Iwaki-Koike-Takei can be generalized to this setting and give the explicit formula, which turns out to be related to the unrefined case by a simple transformation. For an important collection of examples, we write down the quantum curves and find that in the Nekrasov-Shatashvili limit, they take an especially simple form.
\end{abstract}

\maketitle
\setcounter{tocdepth}{2}\tableofcontents

\newpage
\section{Introduction}\label{sec:intro}
In this paper we initiate the extension of several recent results on topological recursion and quantum curves to the \emph{refined} setting. Refined (or ``$\beta$-deformed'') analogues of these ideas were originally pursued by several authors in \cite{CE,CEM1,CEM2,C,M} from the perspective of ($\beta$-deformed) matrix models. For our purposes, such a setting is too restrictive and relies on specific physical models, and thus we reformulate the recursion in a purely geometric way that applies generally to genus zero curves of degree two, before applying it to obtain the corresponding \emph{refined quantum curves}. Thus, we proceed in two steps; in particular,

\begin{enumerate}
    \item We define the refined topological recursion in a uniform and geometric way, avoiding any reference to physical models.  We prove that it satisfies expected properties in analogy with the usual (unrefined) topological recursion. In particular, the output of the refined recursion is a collection of symmetric multidifferentials.
    \item We prove that the corresponding refined quantum curve exists  -- in particular, the $\hbar$-corrections truncate. We obtain an explicit formula, and find that the effect of the refinement can be captured by a simple operation on the unrefined quantum curve of Iwaki-Koike-Takei.
\end{enumerate}

In studying the refined recursion and corresponding quantum curves, a number of surprises appear, some of which we will only touch on in future work. Most pertinent to the present work, we note that in the refined case it turns out that one may introduce additional parameters in the initial data not considered before, without spoiling the fundamental properties. The meaning of these parameters is not entirely clear to us, although their role becomes clearer in the so-called \emph{Nekrasov-Shatashvili limit} \cite{NS1}, which we compute explicitly.


This work lays the ground for several future results. Our motivation came from attempting to generalize the results of \cite{IKT1,IKT2,Iwaki:2021zif,IWAKI2022108191} to the refined setting. In future work we will analyze further the invariants of the refined topological recursion and quantum curves considered here, and relate them to the so-called \emph{(refined) BPS structures} of Bridgeland \cite{Bri18} and Gaiotto-Moore-Neitzke \cite{GAIOTTO2013239}, thus relating the refined topological recursion to refined Donaldson-Thomas theory (generalizing \cite{IWAKI2022108191}). As a result, we expect to be able to provide a solution to Barbieri-Bridgeland-Stoppa's so-called \emph{quantum Riemann-Hilbert problem} for (refined) BPS structures \cite{BBS}, along the lines of \cite{Iwaki:2021zif}. This would provide a canonical example of a ``refined'' \emph{Joyce structure}, recently introduced by Bridgeland to describe the geometry of the space of stability conditions \cite{Bri19}.

Before proceeding to state our main results precisely, let us give a brief background on the topological recursion and motivation for its refined counterpart.

\subsection{Background and motivation}
\subsubsection{Topological recursion}



Topological recursion was introduced by Chekhov-Eynard-Orantin as a recursive technique for solving simple quantum field-theoretic models, so-called Hermitian matrix models. The \emph{partition function} of such models is schematically given by an integral of the form
\begin{equation}
    Z_N = \dfrac{1}{(2\pi)^N N!}\int d\lambda_1\ldots d\lambda_N \prod_{i<j}^{N}({\lambda_i-\lambda_j})^2 e^{-\frac{1}{\hbar}\sum_{i}^N V(\lambda_i)}.
\end{equation}
for a ``potential'' $V$. One major advantage of their method of solution is that it can be applied in generality far beyond the the matrix model setting. Roughly speaking, topological recursion takes as input a \emph{spectral curve}, which is a Riemann surface $\Sigma$ equipped with two meromorphic functions $x,y$ and a certain bidifferential $B$; the output is a infinite collection of meromorphic multidifferentials $\omega_{g,n}$ for $g,n\in \mathbb{Z}_{\geq0}$, via a universal recursive residue formula starting with $\omega_{0,1}:=ydx$ and $\omega_{0,2}:=B$. In particular, all residues involved are taken at the set of \emph{ramification points} $\mathcal{R}$, consisting of certain zeroes and poles of $dx$.

In the decade and a half since its inception, topological recursion has proven to be a powerful tool with applications to quantum field theory, differential equations, and enumerative geometry. In particular, many different enumerative invariants related to the moduli space of curves $\overline{\mathcal{M}}_{g,n}$, such as Hurwitz numbers \cite{BM,BEMS,EMS}, Gromov-Witten invariants \cite{NS2,DN,EO3,DOSS}, Weil-Petersson volumes \cite{Mir1,Mir2,MS1,EO4}, knot invariants \cite{BEM,GS,GJK}, etc. can be cast into and computed using the framework of topological recursion. Most notably, the so-called remodelling conjecture (now theorem) of Bouchard-Klemm-Mari\~{n}o-Pasquetti \cite{BKMP,BKMP2,EO3,FLZ} proves the relation between the Gromov-Witten invariants of toric Calabi-Yau threefolds and $\omega_{g,n}$ where $\Sigma$ is taken to be the ``toric mirror curve''. Recently, a relationship between topological recursion and \emph{Donaldson-Thomas invariants} or \emph{BPS invariants} was found  in certain cases by the first author and collaborators \cite{Iwaki:2021zif,IWAKI2022108191}.

Since topological recursion became known as a powerful tool in mathematics and physics, the structure of the recursion itself has also been investigated intensively. It is shown that the fundamental structure underlying the topological recursion is a set of equations relating different $\omega_{g,n}$'s called \emph{abstract loop equations} \cite{BEO}. 
This approach has made it possible to generalize topological recursion in several directions, e.g., blobbed topological recursion \cite{BS}, topological recursion with higher-order ramifications \cite{BE2,BBCCN}, and supersymmetric topological recursion \cite{SAS,BO2,O2}. Kontsevich and Soibelman proposed an algebraic approach to solve abstract loop equations in terms of so-called Airy structures \cite{ABCD,KS2} which established a new research direction between topological recursion and vertex operator algebras \cite{BBCCN,SAS,BO2,O2,BBCC}.



\subsubsection{Refined topological recursion and $\beta$-deformation}

There is another perspective, more algebraic in nature, which leads us to the refined setting. The partition function of a Hermitian matrix model is annihilated by so-called \emph{Virasoro operators} of the Heisenberg vertex operator algebra (see \cite{BO1,E16,MS} and references therein for this aspect). It turns out that the Heisenberg vertex operator algebra admits a natural one (complex) parameter deformation called $\beta$-deformation, which is known as the linear dilaton CFT, or the free CFT in a background charge in physics literature. 

Taking this perspective over matrix models, one can consider ``partition functions'' which are by definition annihilated by the $\beta$-deformed Virasoro operators. For special values of $\beta$, these can be interpreted as a partition function for a kind of ``$\beta$-deformed matrix model''. When $\beta=1$, we recover the usual Hermitian matrix model, and when $\beta=\frac{1}{2}, 2$ we obtain so-called orthogonal and symplectic matrix models, respectively. However, for generic values of $\beta$, there is often no known interpretation in terms of a matrix model (at least in the traditional sense --- see \cite{DE} for more general examples). For general $\beta$, Dijkgraaf and Vafa \cite{DV} suggested that $\beta$-deformed matrix models explain the Alday-Gaiotto-Tachikawa correspondence \cite{AGT}, that is, the relation between Nekrasov's partition function \cite{Nekrasov:2002qd} and conformal blocks. From the perspective of this relation, our result corresponds to considering the case of generic $\Omega$-background of the gauge theory rather than the self-dual limit $\epsilon_1=-\epsilon_2=\hbar$ which the unrefined topological recursion has access to (see \cite{BMT} for related discussion).
 
The refined topological recursion\footnote{ There is another approach to recursively solve $\beta$-deformed matrix models, which is called \emph{non-commutative topological recursion} \cite{CEM1,CEM2}. Such a recursion is more involved due to the noncommutativity which is present even in the initial data. It would be very interesting to understand the relation to the present work.} was proposed for some special cases of $\beta$-deformed matrix models in \cite{CE,C,M}. However, the analogous theory along the lines of Eynard-Orantin \cite{EO} for arbitrary spectral curves (or even clarification of the initial data) not arising from a matrix model has not yet been established. The main obstacle is that it is not known how to generalise abstract loop equations to the refined setting, which makes it inaccessible from the Airy structure point of view too. We will not solve this problem in general in the present paper --- rather, we restrict to the special case of spectral curves which are of genus zero and degree two. The resulting presence of a global involution is crucial to our formulation, while the condition on the genus can be relaxed. One of the main goals of the present work is to give a clear geometric description of refined topological recursion in this situation, which we hope will enable a deeper understanding of the refined setting more generally in the future.

We will describe all of the new ingredients that appear in the refined recursion below. For now, let us simply note the most salient differences: the $\omega_{g,n}$'s are now indexed by \emph{half-integers} $g\in \frac12 \mathbb{Z}_{\geq0}$, with the property that all $\omega_{g,n}$, $g\notin\mathbb{Z}$ vanish when $\beta=1$, recovering the unrefined case. Furthermore, while the unrefined recursion boils down to the differentials $\omega_{0,1}$ and $\omega_{0,2}$, the refined case requires the introduction of another differential $\omega_{\frac12,1}$ constructed from the initial data with additional parameters ${\bm \mu}$. Finally, we note that the pole structure of $\omega_{g,n}$ is more complicated in the refined case, which is why it is no longer sufficient to take residues only at ramification points as in the unrefined case, complicating various analyses.

\subsubsection{Quantum curves}
The main application of (refined) topological recursion relevant to the present paper is the construction of quantum curves. These are differential operators which annihilate a certain formal $\hbar$-series constructed from the topological recursion.

More precisely, in both the refined and unrefined cases, we may carefully assemble the $\omega_{g,n}$ in a specific way to obtain the \emph{topological recursion wavefunction} $\psi^{{\rm TR}}(x)$:
\begin{equation}
    \psi^{\rm TR}(x) := \exp \left( \sum_{g\in \frac{1}{2}\mathbb{Z}_{\geq 0},n>0} \dfrac{\hbar^{2g-2+n}}{\beta^{n/2}} \dfrac{1}{n!} \int_{ D(z;{ {\bm \nu}})}\!\!\!\ldots\int_{D(z;{ {\bm \nu}})}\omega_{g,n}(\zeta_1,\ldots,\zeta_n)\right)
\end{equation}
where $\beta=1$ in the unrefined case. The multi-integrals are determined by a divisor $D(z;{\bm \nu})$ which depends on some free parameters ${\bm \nu}$. If there exists a differential operator that annihilates the wavefunction of the form
\begin{equation}
\left(\hbar^{2}\dfrac{d^2}{dx^2}+q(x,\hbar)\hbar \dfrac{d}{dx}+r(x,\hbar)\right)\psi^{\rm \mathsmaller{TR}}(x)=0,
\end{equation} 
then such an equation is called the \emph{quantum curve} associated to the spectral curve.

The study of quantization via topological recursion in the unrefined setting has been investigated in varying levels of generality in e.g. \cite{GS,BE2,Iwaki19,EGMO}. The notion of quantum curve itself drew attention in \cite{ADKMV} which discussed a relation between integrable systems and the topological string. It was then proposed by \cite{DM1,DM2} that quantum curves are closely related to the geometry of Higgs bundles (see \cite{DM3} for a summary of progress on quantum curves). 

Quantizing curves in general is highly nontrivial. In the case of curves of genus zero and degree two, Iwaki-Koike-Takei \cite{IKT1,IKT2} were able to show existence of and obtain explicit expressions for quantum curves. Furthermore, they used the form of the quantum curves to obtain closed formulae for certain topological recursion invariants called \emph{free energies} in a collection of nine explicit examples related to (confluent limits of) the Gauss hypergeometric equation, so-called spectral curves of \emph{hypergeometric type}. Building on this work, Iwaki and the first author \cite{Iwaki:2021zif, IWAKI2022108191} recently found a relation between the (unrefined) topological recursion free energy and the corresponding \emph{BPS structure} which encodes Donaldson-Thomas invariants corresponding to stability conditions on a certain Calabi-Yau-3 triangulated category arising from such curves \cite{Bri18,bridgeland2015quadratic}.

\subsection{Main results}
In the present paper, we formulate and establish the refined topological recursion and determine the corresponding quantum curves in the case where $\Sigma$ is of genus zero and degree two (thus equipped with a global involution $\sigma$).

Our first achievement is thus a definition --- we define the refined topological recursion from the initial data of a \emph{refined spectral curve} $\mathcal{S}^{\bm \mu}$, which in addition to the usual data of $(\Sigma,x,y,B)$ includes a tuple ${\bm \mu}$ of complex parameters associated to a choice of half of the poles and zeroes of $y$. The \emph{refined topological recursion} is then given by a modified recursion formula, whose starting point now includes an additional 1-form $\omega_{\frac12,1}$ we describe in Section {\ref{subsubsec:RTR}}. When all the parameters in ${\bm \mu}$ are set to $+1$ or $-1$, we recover the results appearing in \cite{CE,C,M} for $\beta$-deformed matrix models.

As our first main theorem, we establish the fundamental properties of this recursion, the first and third of which are identical to the unrefined case:
\begin{theorem}[Theorem \ref{thm:RTR}]\label{thm:RTRintro}
Let $\mathcal{S}^{\bm \mu}$ be a genus zero degree two refined spectral curve satisfying {\rm Assumption \ref{ass:tr}}. For any $2g+n\geq2$, the multidifferentials constructed from the refined topological recursion satisfy the following properties:
\begin{enumerate}[label={\bf RTR\arabic*:}, ref=\arabic*,leftmargin=1.7cm]
\item The multidifferential $\omega_{g,n+1}$ is symmetric;

\item All poles of $\omega_{g,n+1}(p_0,J)$ in the first variable lie in $\mathcal{R}^*\cup\sigma(J)$;

\item At any $o\in \Sigma$, $\omega_{g,{n+1}}$ is residue-free in the first (thus, any) variable:
\begin{equation}
    \underset{p=o}{{\rm Res}}\,\omega_{g,n+1}(p,J)=0;
\end{equation}
\vspace{-4mm}
\item  For $n>0$, we have
\vspace{-3mm}
\begin{equation}
(2g-2+n)\,\omega_{g,n}(J)=-\left(\sum_{r\in\mathcal{R}^*}\underset{p=r}{{\rm Res}}+\sum_{r\in\sigma(J)}\underset{p=r}{{\rm Res}}\right)\Phi(p)\cdot\omega_{g,n+1}(p,J).
\end{equation}

\end{enumerate}
where $J=(p_1,\ldots,p_n)\in \Sigma^n$, $\mathcal{R}^*$ is a certain subset of $\mathcal{R}$, and $\Phi$ is any primitive of $ydx$.
\end{theorem} 

From the formula defining the recursion, the symmetry is completely non-obvious, and indeed takes considerable work to prove. The main difference from the unrefined case is the presence of additional poles in $\omega_{g,n+1}(p_0,J)$ at $p_0$ and $J$, and correspondingly we must take care at these points throughout. Nonetheless, $\omega_{g,n+1}$ remain residueless in the end. 



{Note that since our recursion is not restricted to  those arising from explicit physical models appearing in \cite{CE,C,M}, it can be applied to degree two curves that do not admit a $\beta$-deformed matrix model description such as the Airy curve $y^2-x=0$ or the (degenerate) Bessel curve $xy^2-1=0$ (both of which play a fundamental role in the theory), which appear as examples in Section {\ref{sec:QCHG}}.}

Having established the fundamental properties of the recursion, we turn to the quantization. We show that the explicit construction of the quantum curve in \cite{IKT1,IKT2} generalizes to our setting. Furthermore, we observe that the final result can in fact be obtained by a remarkably simple operation on the original unrefined quantum curve. More precisely, we have:

\begin{theorem}[Theorem \ref{thm:main}] Given a genus zero degree two refined spectral curve $\mathcal{S}^{\bm \mu}$ satisfying {\rm Assumption} \emph{\ref{ass:tr}} and a divisor
\begin{equation}
    D(z;{\bm \nu })=[z]-\sum_{p\in \mathcal{P}}\nu_{p}[p], \qquad \sum_{p \in \mathcal{P}}\nu_p=1,
\end{equation}
the wavefunction $\psi^{\rm \mathsmaller{TR}}(x)$
is a formal solution of the equation
\begin{equation}
\left(\hbar^{2}\dfrac{d^2}{dx^2}+q(x,\hbar)\hbar \dfrac{d}{dx}+r(x,\hbar)\right)\psi^{\rm \mathsmaller{TR}}(x)=0
\end{equation}  
where $q(x,\hbar)$, $r(x,\hbar)$ are (finite) formal $\hbar$-series of rational functions
\begin{equation}
    q(x,\hbar)=q_0(x)+\hbar q_1(x), \quad r(x,\hbar) = r_0(x) + \hbar r_1(x) + \hbar^2 r_2(x).
\end{equation}
which can be explicitly computed. Furthermore, if we denote the unrefined limit $\beta=1$ by $r(x,\hbar;{\bm \nu})=r(x,\hbar;{\bm \mu},{\bm \nu},\beta)|_{\beta=1}$, $q(x,\hbar;{\bm \nu})=q(x,\hbar;{\bm \mu},{\bm \nu},\beta)|_{\beta=1}$, we have
\begin{align}
    &r(x,\hbar;{\bm \mu},{\bm \nu},\beta)=\frac{1}{\beta}r(x,{\hbar \beta^{-\frac12}};\widehat{{\bm \nu}})\\
    &q(x,\hbar;{\bm \mu},{\bm \nu},\beta)=\frac{1}{\beta^{\frac12}}q(x,{\hbar {\beta^{-\frac12}}};{{\bm \nu}}),
\end{align}
where $\widehat{\bm \nu}$ denotes the shifted parameters
\begin{equation}
    \hat{\nu}_p={\nu_p}-{\rm ord}_p(y)\left(\tfrac{1+\mu_p}{2}\right)(\beta-1).
\end{equation}
\end{theorem}
We note that for a few special cases -- at least the Weber and the Kummer curves -- the quantum curve was obtained through CFT methods in \cite{MS,CHMS}. However, that construction did not involve a choice of integration divisor or the parameters ${\bm \mu}$, and corresponds to a particular value of ${\bm \mu}$ and ${\bm \nu}$ in our notation. In particular, we construct the whole family of quantum curves in a uniform geometric way, which is in fact essential for applications we have in mind.

In the last section, we write down explicitly the quantum curves for the spectral curves of hypergeometric type which played a fundamental role in \cite{IKT1,IKT2,Iwaki:2021zif,IWAKI2022108191}. We furthermore compute their Nekrasov-Shatashvili limit, in which we take $\hbar\to0,\beta\to\infty$ while keeping $\hbar\beta^{\frac12}=\epsilon_1$ fixed. This limit enjoys a close relation to quantum integrability \cite{NS1,Nekrasov:2011bc,Aganagic:2011mi} and many other applications (e.g. \cite{hollandskidwai,Hollands:2019wbr,Hollands:2021itj,Alim:2022oll,Grassi:2019coc,Grassi:2021wpw}). Thus, in the special case we have considered, our result relates different notions of quantum curve appearing in the literature. We hope our result will help shed light on the relationship between the two limits more generally in the future.

\subsection{Comments and open problems}
Let us comment on our results and mention some problems we hope to address in the future.

\begin{itemize}
\itemsep2pt
    \item One of the advantages of the geometric description of the refined topological recursion is that we can now ask enumerative questions in the refined setting. In particular, it is known that the multidifferentials $\omega_{g,n}$ of the unrefined topological recursion for the Airy curve $y^2-x=0$ are generating functions of the Kontsevich-Witten intersection numbers over moduli spaces of Riemann surfaces $\overline{\mathcal{M}}_{g,n}$ \cite{EO2,K,W}. It is an open question whether the refined topological recursion on the Airy curve admits an enumerative interpretation. The matrix model argument suggests that, for a special value of $\beta$, it could be related to moduli spaces of non-orientable surfaces.
    
    \item Although our computations are limited to genus zero curves, the refined topological recursion can be generalised to higher genus degree two curves, where additional care must be taken due to the presence of non-contractible cycles. The higher genus recursion in the refined setting is expected to have interesting applications in the context of (quantum) Painlev\'{e} equations and isomonodromy systems. In particular, the relation between the unrefined topological recursion and isomonodromy systems has been studied in \cite{Iwaki19,IM,IS}, and it would be interesting to investigate how the refinement contributions appear in their results. 
    
    \item Refinement of the topological recursion for higher degree algebraic curves is a tricky problem. This is due to the absence of the global involution which is critical to our construction. Equivalently from the Airy structure point of view, twisted modules of the Heisenberg vertex operator algebra are known not to admit $\beta$-deformation. Nevertheless, our refined recursion for the Airy curve is well defined, hence it may be possible to reformulate it algebraically along these lines.
    
    \item K. Iwaki and the first author \cite{Iwaki:2021zif} recently found a relation between the (unrefined) topological recursion free energy and the corresponding \emph{BPS structure} which encodes Donaldson-Thomas invariants arising from hypergeometric type spectral curves. In a sequel paper, we will prove that an analogous relation continues to hold in the refined setting, but with new ingredients due to the refinement.
    
    \item In \cite{IWAKI2022108191}, the above result was used to solve a certain Riemann-Hilbert problem defined by Bridgeland \cite{Bri18}, which describes the geometry of certain spaces of stability conditions. In that work it was shown that Voros coefficients (invariants of the quantum curve) associated with the unrefined examples of \cite{IKT2} provide a canonical solution, and furthermore the associated \emph{{$\tau$}-function} was computed. We expect a corresponding relation between the refined quantum curves and the corresponding \emph{quantum} Riemann-Hilbert problem of Barbieri-Bridgeland-Stoppa \cite{BBS}, whose details will appear in a future publication.
    \item While we obtained the generalized recursion with new parameters ${\bm \mu}$, we have not understood deeply their meaning. We hope that a careful study of the above problems may shed light on this.

\end{itemize}

\subsubsection*{Acknowledgements}

We thank Tom Bridgeland, Leonid Chekhov, Lotte Hollands, Kohei Iwaki, Masahide Manabe, Piotr Su{\l}kowski, and Alex Youcis for helpful discussions and correspondence. The work of OK is supported by a JSPS Postdoctoral Fellowship for Research in Japan (Standard) and a JSPS Grant-in-Aid (KAKENHI Grant Number 19F19738). The work of KO is in part supported by the Engineering and Physical Sciences Research Council under grant agreement ref.~EP/S003657/2, in part by the TEAM programme of the Foundation for Polish Science co-financed by the European Union under the European Regional Development Fund (POIR.04.04.00-00-5C55/17-00), and also in part by a JSPS Grant-in-Aid (KAKENHI Grant Number 21H04994 and 22J00102). A significant part of the work of KO was carried out while he was affiliated at the University of Sheffield and the University of Warsaw.

\section{Refined topological recursion}\label{sec:TR}

In this section, we define geometrically and establish the fundamental properties of the refined topological recursion for our class of spectral curves. We first review the standard (unrefined) topological recursion before proceeding to the refined case.

\subsubsection*{Multidifferentials}
The topological recursion involves a collection of geometric objects called multidifferentials defined on the product $(\Sigma)^k$ of a Riemann surface $\Sigma$.
Denote by $\pi_i:(\Sigma)^k\rightarrow \Sigma$, $i=1,\ldots k$ the $i$th projection map. Then a \emph{$k$-differential} is a meromorphic section of the line bundle
\begin{equation}\label{eq:multidif}
\pi_1^*T^*\Sigma\otimes \pi_2^*T^*\Sigma\ldots\otimes\pi_k^*T^*\Sigma. 
\end{equation}

We use the term \emph{multidifferential} to mean a $k$-differential without specifying $k$, and sometimes \emph{bidifferential} for the case $k=2$. Usually when writing multidifferentials in coordinates $z_1,\ldots z_k$ we simply omit the $\otimes$ symbol, and write e.g. $f(z_1,z_2) dz_1dz_2$. Integration and differentiation is defined factorwise in the obvious way. A multidifferential will be called \emph{symmetric} if it is invariant under permutation of variables:
\begin{equation}
\omega(z_1,\ldots,z_k)=\omega(z_{s(1)}, \ldots, z_{s(k)})
\end{equation}
for any permutation $s\in\mathfrak{S}_k$.

We frequently use the letter $J=(p_1,\ldots,p_n)\in \Sigma^n$ to denote a tuple of points, and write e.g. $\omega(p_0,J)$ to denote $\omega(p_0,p_1,\ldots,p_n)$. We also use $J_0$ to denote $(p_0,\ldots,p_n)$. Furthermore, since the multidifferentials we consider are are symmetric, we often abuse notation and write $J$ to denote the \emph{set} of (possibly coincident) points $\{p_1,\ldots,p_n\}$ as a subset of $\Sigma$.

Finally, we often use the word \emph{residue-free} to mean that a given one-form (or a multidifferential considered in a fixed variable) has zero residue at every point.

\subsection{Unrefined topological recursion}

The initial data of the \emph{unrefined} topological recursion, i.e., the standard Chekhov-Eynard-Orantin topological recursion \cite{EO,EO2,CEO}\footnote{In the present paper, we call the Chekhov-Eynard-Orantin topological recursion the unrefined topological recursion in order to emphasise comparisons with the refined topological recursion.}, is a so-called spectral curve which is a Riemann surface equipped with additional geometric data. 

\begin{definition}\label{def:curve}
A \emph{(global) spectral curve} $\mathcal{S}$ is a quadruple $\mathcal{ S}=(\Sigma,x,y,B)$ where
\begin{itemize}[label=\raisebox{0.25ex}{\tiny$\bullet$}]
\itemsep2pt
\item $\Sigma$ is a compact Riemann surface
\item $x,y:\Sigma\rightarrow\mathbb{P}^1$ are two meromorphic functions such that $dx$ and $dy$ do not vanish simultaneously.
\item $B$ is a symmetric bidifferential with a double pole on the diagonal with biresidue 1.\footnote{This means that in {any} local coordinate $z$, $B(z_1,z_2)=\frac{dz_1dz_2}{(z_1-z_2)^2}+{O}(1)$ in the limit $z_1\to z_2$.}
\end{itemize}
\end{definition}

$B$ is usually chosen to be the so-called \emph{Bergman kernel}, or the \emph{fundamental bidifferential} which is canonically associated to a Riemann surface with Torelli marking (see e.g. \cite{EO}). We call any zero of $dx$ or pole of order 3 or higher of $dx$ a \emph{ramification point}, and denote by $\mathcal{R}$ the set of all ramification points. We will furthermore make the assumption that all ramification points $r$ are \emph{simple}, so that $x$ is of degree two at $r$\footnote{See e.g. \cite{BE1} for spectral curves which violate the assumption and the corresponding recursion.}. In a small neighbourhood $U_r\subset\Sigma$ of each ramification point $r\in\mathcal{R}$, there exists a canonical local involution map $\sigma_r$ such that $x(p)=x(\sigma_r(p))$ and $y(p)\neq y(\sigma_r(p))$ where $p\in U_r$. Then for each ramification point, we define the associated \emph{recursion kernel} as
\begin{equation}
{\rm K}_r(p_0,p):=\frac12\frac{\int_{\sigma_r(p)}^{p}B(p_0,\cdot)}{\Delta y(p)\cdot dx(p)},\label{kernel}
\end{equation}
where $p_0\in\Sigma$ and $p\in U_r\subset\Sigma$, and we write \begin{equation}
\Delta y(p):=y(p)-y(\sigma_r(p)).
\end{equation}
Note that the recursion kernel is uniquely determined near each ramification point, but only locally defined in general. Finally, we call a ramification point $r\in\mathcal{R}$ \emph{ineffective} if $\Delta y\!\cdot \!dx$ is singular at $r$, and \emph{effective} otherwise. The set of effective ramification points is denoted by $\mathcal{R}^*\subset \mathcal{R}$.

In order to define the recursion, various additional conditions are typically imposed on the spectral curve in the literature. In  the present paper, we will assume the following\footnote{For the higher degree case, see e.g. \cite{BBCCN,BE1}}:
\begin{assumption}
\label{ass:tr}
\,
\begin{itemize}[label=\raisebox{0.25ex}{\tiny$\bullet$}]
    \item The function field $\mathbb{C}(\Sigma)$ is generated by $x$ and $y$,
    \item  All ramification points $r\in\mathcal{R}$ are simple i.e. $x$ is of degree $2$ at $r$,
    \item For any two distinct ramification points $r_1, r_2$, $x(r_1)\neq x(r_2)$,
    \item $\Delta y \!\cdot\! dx$ has at most a double zero at each ramification point.
\end{itemize}
\end{assumption}

The unrefined topological recursion \cite{EO,EO2,CEO} provides a recipe for computing an infinite sequence of multidifferentials $\omega_{g,n+1}$ on $(\Sigma)^{n+1}$ for all $g,n\in\mathbb{Z}_{\geq0}$ with $2g+n\geq2$. For an appropriate choice of a spectral curve, these $\omega_{g,n+1}$ become generating functions of some enumerative invariants such as Gromov-Witten invariants. We refer to the readers \cite{E} for a summary of achievements of the topological recursion. Let us now give the definition:

\begin{definition}[\cite{CEO,EO,EO2}]\label{def:CEO}
Given a spectral curve $\mathcal{S}$ satisfying Assumption \ref{ass:tr}, the \emph{(unrefined) topological recursion} is a recursive construction of an infinite sequence of meromorphic multidifferentials $\omega_{g,n+1}$ on $(\Sigma)^{n+1}$ for $g,n\in\mathbb{Z}_{\geq0}$ with $2g+n\geq0$ by the initial data:
\begin{equation}
    \omega_{0,1}(p_0):=y(p_0)dx(p_0) \qquad \omega_{0,2}(p_0,p_1):=B(p_0,p_1),
\end{equation}
together with the following recursive formula:

\begin{align}
\omega_{g,n+1}(p_0,J)&=\sum_{r\in\mathcal{R}}\,\underset{p=r}{\text{Res}}\,\,{\rm K}_r(p_0,p)\cdot\text{Rec}_{g,n+1}(p,J),
\end{align}
where
\begin{align}
\text{Rec}_{g,n+1}(p,J)&=\omega_{g-1,n+2}(p,\sigma_r(p),J)\, +\!\!\!\sum^{*}_{\substack{g_1+g_2=g \\ J_1\sqcup J_2=J}}\omega_{g_1,|J_1|+1}(p,J_1)\cdot\omega_{g_2,|J_2|+1}(\sigma_r(p),J_2), \label{TR}
\end{align}
with $J=(p_1,...,p_n)$ denoting a point in $(\Sigma)^n$, and the $*$ in the sum means that we remove terms involving $\omega_{0,1}$.
\end{definition}
\begin{remark}Note that the formula is written in terms of residues around ramification points, which can equivalently be thought of as an integral along an appropriate contour. Thus, the topological recursion formula can be thought of as genuine integration of ${\rm Rec}_{g,n}$ against a kernel. More generally, we will occasionally think of and refer to expressions inside residues as integrands. 
\end{remark}
In their seminal work, Eynard-Orantin established several fundamental properties of the $\omega_{g,n+1}$:

\begin{theorem}[\cite{EO}]\label{thm:TR}
For any $2g+n\geq2$, the multidifferentials constructed from the unrefined topological recursion satisfy the following properties:
\begin{enumerate}[label={\bf TR\arabic*:}, ref=\arabic*]
\itemsep2pt
\item\label{TR1} The multidifferential $\omega_{g,n+1}$ is symmetric;
\item\label{TR2} All poles of $\omega_{g,n+1}$ (in any variable) lie in $\mathcal{R}^*\subset \mathcal{R}$;
\item\label{TR3} At any $o\in \Sigma$, $\omega_{g,{n+1}}$ is residue-free in the first (thus, any) variable:
\begin{equation}
    \underset{p=o}{{\rm Res}}\,\omega_{g,n+1}(p,J)=0;
\end{equation}
\item \label{TR4} For $n>0$, we have
\begin{equation}
(2g-2+n)\,\omega_{g,n}(J)=-\sum_{r\in\mathcal{R}^*}\underset{p=r}{{\rm Res}}\,\Phi(p)\cdot\omega_{g,n+1}(p,J).\label{inverse}
\end{equation}
\end{enumerate}
\end{theorem}

TR\ref{TR1} is in fact nontrivial at first glance from the recursion formula \eqref{TR} because the $p_0$-dependence is of a different nature than the other variables $J$, but it is proven\footnote{Another, more general, proof is to consider it in a relatively new formalism, so-called Airy structures proposed by Kontsevich and Soibelman \cite{KS2} (see also \cite{ABCD}).} in \cite[Theorem 4.6]{EO}. TR\ref{TR2} and TR\ref{TR3} are almost straightforward from the recursion \eqref{TR}, and TR\ref{TR4} is proven in \cite[Theorem 4.7]{EO}. However, as we will see below, TR\ref{TR2} no longer holds in the refined setting, and the remaining properties require a more careful analysis as a result. 

In light of TR\ref{TR4}, it is natural to wonder what happens if we consider \eqref{inverse} with $n=0$ and $g>1$, that is, to compute $\omega_{g,0}$. The right hand side is still a well-defined invariant, and this motivates the definition of the \emph{free energy} as follows \cite{EO}:
\begin{definition}\label{def:Fg}
For $g\geq 2$, the \emph{genus g free energy $F_g$ of the unrefined topological recursion}\footnote{$F_0$ and $F_1$ are defined in different ways \cite{CEO,EO,EO2}.} is defined by
\begin{equation}
F_g=\omega_{g,0}:=\frac{1}{2-2g}\sum_{r\in\mathcal{R}^*}\underset{p=r}{\text{Res}}\,\Phi(p)\!\cdot\omega_{g,1}(p),\label{TRFg}
\end{equation}
where $\Phi$ is any primitive of $ydx$.
\end{definition} 
\noindent Note that TR\ref{TR3} ensures that $F_g$ is indeed independent of the choice of primitive $\Phi$.

\subsection{Refined topological recursion}

\subsubsection{Refined spectral curves and the refined recursion}\label{subsubsec:RTR}

In this paper, we consider a more specific setting than in the previous section. To define the refined topological recursion, we will first specialize to a nice class of curves (in fact, we will presently only consider the case in which the underlying curve $\Sigma$ has genus zero). As before, $\Sigma$ should be equipped with two functions $x$ and $y$. Then in addition to Assumption \ref{ass:tr}, we will impose that
\begin{assumption}
\label{ass:rtr}
 \,
\begin{itemize}[label=\raisebox{0.25ex}{\tiny$\bullet$}]
    \item The meromorphic functions $x$ and $y$ satisfy $y^2 = Q(x)$ for rational $Q$ which is not a perfect square.
    \item $y dx$ does not vanish anywhere on $\Sigma\setminus\mathcal{R}$.
\end{itemize}
\end{assumption}

Thanks to the first statement
, $\Sigma$ is an irreducible algebraic curve defined by $y^2-Q(x)=0$. In particular the involution operator $\sigma$ is \emph{globally} defined by exchanging sheets. The map $x:\Sigma\mapsto\mathbb{P}^1$ forms a branched double cover of a ``base curve'' $\mathbb{P}^1$, and we define and denote ramification points in the same manner as before. 

\begin{remark}\label{rem:ass}
More generally, our results hold replacing $y^2-Q(x)$ by any irreducible polynomial $P(x,y)$ of degree two in $y$, as long as the genus is zero and the remaining assumptions are satisfied. In this case, $y$ should be replaced by $\frac{1}{2}\Delta y$ at various points throughout the text, together with slight extra care for the quantum curve (see Remark \ref{rem:generalcaseqc}). For simplicity of notation, and since any degree two curve can be brought into the form of Assumption \ref{ass:rtr}, we assume this form henceforth.
\end{remark}

 We let $\mathcal{P}$ denote the set of poles of $ydx$, and let $\widetilde{\mathcal{P}}$ be the set of all zeroes and poles of $y$ that are not in $\mathcal{R}$ --- in particular, under Assumption \ref{ass:tr}, $\widetilde{\mathcal{P}}\subset \mathcal{P}$. Thanks to the involution, $\mathcal{P}$ carries a $\mathbb{Z}_2$ action, which is free when restricted to $\widetilde{\mathcal{P}}$. Thus $\widetilde{\mathcal{P}}$ can be (non-uniquely) decomposed into ``positive'' and ``negative'' subsets $\widetilde{\mathcal{P}}_{\pm}$ by choosing an element from each orbit. We fix a choice of such a decomposition as part of our initial data, and to each $p\in \widetilde{\mathcal{P}}_+$ we assign a parameter $\mu_p\in\mathbb{C}$. It is convenient to assemble them into one package as a (complex) divisor $D({\bm \mu})$ given by:

\begin{equation}
    D({\bm \mu}):=\sum_{p\in\widetilde{\mathcal{P}}_+}\mu_p[p]\label{deta}
\end{equation}
It turns out that this divisor is an additional degree of freedom in the initial data for the recursion, and $\omega_{g,n+1}$ in general depend on ${\bm \mu}$. By convention, we also define $\mu_{p}:=-\mu_{\sigma(p)}$ for $p\in\widetilde{\mathcal{P}}_-$.

\begin{definition}\label{def:Rcurve}
A \emph{genus zero degree two refined spectral curve} $\mathcal{S}^{\bm \mu}$ is a quintuple \newline $\mathcal{S}^{\bm \mu}=(\Sigma,x,y,B,D({\bm \mu}))$ such that:
\begin{itemize}[label=\raisebox{0.25ex}{\tiny$\bullet$}]
    \item The compact Riemann surface $\Sigma=\mathbb{P}^1$, 
    \item Meromorphic functions $x,y:\Sigma\rightarrow \mathbb{P}^1$ whose differentials $dx$, $dy$ do not simultaneously vanish, and satisfying Assumption \ref{ass:rtr},
    \item $B$ is the unique symmetric bidifferential whose only pole is a double pole on the diagonal with biresidue 1,
    \item A choice of complex divisor $D({\bm \mu})$ supported at a choice of $\widetilde{\mathcal{P}}_+$.
\end{itemize}
\end{definition}

\begin{remark} It is not much more difficult to define the higher genus case, but we avoid it in the present work, since proving the analogue of Theorem \ref{thm:RTR} below is more involved. We leave the treatment of such curves to future work.\end{remark}

We are now ready to define the refined topological recursion. Its first appearance was in Chekhov-Eynard \cite{CE} in the context of $\beta$-deformed matrix models, and several papers such as \cite{BMS,MS,C,M} discuss its properties and applications. Our formulation is purely in terms of (genus zero, degree two) refined spectral curves. We will prove that all the expected properties (with slight modifications to TR\ref{TR2} and TR\ref{TR4}) continue to hold for genus zero degree two refined spectral curves in the coming sections. We note that although we take $\beta\in\mathbb{C}^*$, it is only the combination $\mathscr{Q}=\beta^{\frac12}-\beta^{-\frac12}$ which appears in the recursion itself, so we write everything in terms of $\mathscr{Q}$ until necessary, when we define the wavefunction in Section \ref{sec:QC}.

\begin{definition}\label{def:RTR}
Fix $\mathscr{Q}\in\mathbb{C}$. Given a genus zero degree two refined spectral curve $\mathcal{S}^{\bm \mu}=(\Sigma,x,y,B,D({\bm \mu}))$ satisfying Assumption \ref{ass:tr}, the \emph{refined topological recursion} is a recursive construction of an infinite sequence of multidifferentials $\omega_{g,n+1}$ for $2g,n\in\mathbb{Z}_{\geq0}$ with $2g+n\geq2$ from the initial data:
\allowdisplaybreaks[0]
\begin{equation}
    \omega_{0,1}(p_0):=y(p_0)dx(p_0), \qquad \omega_{0,2}(p_0,p_1):=B(p_0,p_1),
\end{equation}
\begin{equation}\label{eq:half1def}
    \omega_{\frac12,1}(p_0):=-2\mathscr{Q}\left(\sum_{r\in\mathcal{R}}\underset{p=r}{\text{Res}}+\underset{p=\sigma(p_0)}{\text{Res}}+\sum_{r\in\widetilde{\mathcal{P}}_+}\mu_r\cdot\underset{p=r}{ \text{Res}}\,\right){\rm K}(p_0,p)\! \cdot\! dx(p)\! \cdot \! dy(p).
\end{equation}
\allowdisplaybreaks[1]

\noindent and the following formula:

\begin{equation}
\omega_{g,n+1}(p_0,J)=-2\left(\sum_{r\in\mathcal{R}}\underset{p=r}{\text{Res}}+\sum_{r\in\sigma(J_0)}\underset{p=r}{\text{Res}}+\sum_{r\in\widetilde{\mathcal{P}}_+}\underset{p=r}{\text{Res}}\right){\rm K}(p_0,p)\!\cdot\text{Rec}_{g,n+1}^{\mathscr{Q}}(p,J),\label{eq:RTR}
\end{equation}


\noindent where
\begin{align}
    &\text{Rec}^\mathscr{Q}_{0,3}(p,p_1,p_2)&=&\;\;\frac12\Delta\omega_{0,2}(p,p_1)\cdot \Delta\omega_{0,2}(p,p_2),\label{Rec0,3}\\
&\text{Rec}^\mathscr{Q}_{\frac12,2}(p,p_1)&=&\;\;\Delta\omega_{0,2}(p,p_1)\cdot \omega_{\frac12,1}(p)-\mathscr{Q}\,dx(p)\cdot d_p\left(\frac{ \omega_{0,2}(\sigma(p),p_1)}{dx(p)}\right),\label{Rec1/2,2}\\
&\text{Rec}^\mathscr{Q}_{1,1}(p)&=&\;\;\omega_{\frac12,1}(p)\cdot\omega_{\frac12,1}(p)-\omega_{0,2}(p,\sigma(p))+\mathscr{Q}\, dx(p)\cdot d_{p}\left(\frac{\omega_{\frac12,1}(p)}{dx(p)}\right),\label{Rec1,1}
\end{align}

\noindent and for $2g+n\geq3$

\begin{align}
\text{Rec}^\mathscr{Q}_{g,n+1}(p,J)=&\sum_{i=1}^n\Delta\omega_{0,2}(p,p_i) \cdot \omega_{g,n}(p,\widehat{J}_i)+\sum^{**}_{\substack{g_1+g_2=g \\ J_1\sqcup J_2=J}} \omega_{g_1,n_1+1}(p,J_1) \cdot  \omega_{g_2,n_2+1}(p,J_2) \nonumber\\
&+\omega_{g-1,n+2}(p,p,J)+\mathscr{Q}\,dx(p)\cdot d_p\left(\frac{ \omega_{g-\frac12,n+1}(p,J)}{dx(p)}\right),\label{RecQ}
\end{align}

\noindent where $J:=(p_1,...,p_n)$ denotes a point in $(\Sigma)^n$, $J_0=(p_0,p_1\ldots,p_n)\in(\Sigma)^{n+1}$, and $\sigma(J_0):=(\sigma(p_0),\sigma(p_1),...,\sigma(p_n))$, $\widehat{J}_i:=J\backslash \{p_i\}$, the $**$ in the sum means that we remove terms involving $\omega_{0,1}$ or $\omega_{0,2}$, $d_p$ is the exterior derivative with respect to $p$, and 
\begin{equation}
    \Delta\omega_{0,2}(p,p_i):=\omega_{0,2}(p,p_i)-\omega_{0,2}(\sigma(p),p_i).
\end{equation}
\end{definition}
\begin{remark}Let us make a few comments about the differences between the unrefined and the refined topological recursion. First, we emphasize that the refinement parameter $\mathscr{Q}$ does not appear in the definition of refined spectral curve, and it becomes necessary only when we define the refined recursion formula \eqref{eq:RTR}. Second, there is no involution operator $\sigma$ in $\text{Rec}_{g,n+1}^\mathscr{Q}(p,J)$ above unlike $\text{Rec}_{g,n+1}(p,J)$ in \eqref{TR}, and we note that this is not a typo.
Third, $\text{Rec}^\mathscr{Q}_{g,n+1}(p,J)$ with $2g+n=2$ are slightly different from $\text{Rec}^\mathscr{Q}_{g,n+1}(p,J)$ with $2g+n>2$ due to a non-universal behaviour of $\omega_{0,2}$. Finally and most importantly, the recursion formula does not boil down to residue computations at ramification points $\mathcal{R}$ in contrast to the unrefined recursion formula \eqref{TR}, but rather we need to additionally take residues at $\sigma(J_0)$ and at $\widetilde{\mathcal{P}}_+$.
\end{remark}

Although it will not play a role in the rest of the present paper, we note that we may define the free energy by the exact same formula as the unrefined case:
\begin{definition}\label{def:rFg}
For $g\geq 2$, the \emph{genus g free energy $F_g$ of the refined topological recursion} is defined by
\begin{equation}
F_g=\frac{1}{2-2g}\sum_{r\in\mathcal{R}^*}\underset{p=r}{\text{Res}}\,\Phi(p)\!\cdot\omega_{g,1}(p),\label{RTRFg}
\end{equation}
where $\Phi$ is any primitive of $ydx$.
\end{definition} 
\noindent The fact that $F_g$ is well-defined (independent of the choice of $\Phi$) will follow from the refined analogue of TR\ref{TR3} proved in the next section. 

\subsubsection{Reduction to the unrefined case}
Due to our slightly different formulation of the refined case, our first task is to show that when $\mathscr{Q}=0$, the refined topological recursion boils down to the unrefined topological recursion. We introduce the following lemma which we use several times in the forthcoming sections:
\begin{lemma}\label{lem:contour}
Given a compact Riemann surface $\Sigma$ equipped with an involution $\sigma:\Sigma\rightarrow\Sigma$, let $\mathscr{R}$ denote the set of fixed points of $\sigma$. Suppose $\omega$ is any meromorphic differential invariant under $\sigma$, and let ${\mathscr{P}}$ be the set of all poles of $\omega$ that are not in $\mathscr{R}$. $\mathscr{P}$ admits a (non-unique) decomposition ${\mathscr{P}}={\mathscr{P}}_+\sqcup{\mathscr{P}}_-$. Given such a choice,   
we have
\begin{equation}
    2\left(\sum_{r\in\mathscr{R}}\underset{p=r}{{\rm Res}}+\sum_{r\in{\mathscr{P}}_+}\underset{p=r}{{\rm Res}}\right)\omega(p)=\sum_{r\in\mathscr{R}}\underset{p=r}{{\rm Res}}\,\omega(p).
\end{equation}
\end{lemma}
\begin{proof}

Since $\omega(p)$ is invariant under the involution $\sigma$, we have:
\begin{equation}
    \sum_{r\in{\mathscr{P}}_+}\underset{p=r}{{\rm Res}}\,\omega(p)=\sum_{r\in{\mathscr{P}}_-}\underset{p=r}{{\rm Res}}\,\omega(p).
\end{equation}
Furthermore, the sum of all residues of a meromorphic differential is zero, hence
\begin{equation}
    \sum_{r\in\mathscr{R}}\underset{p=r}{{\rm Res}}\,\omega(p)+2\sum_{r\in{\mathscr{P}}_+}\underset{p=r}{{\rm Res}}\,\omega(p)=0,
\end{equation}
which gives the result.
\end{proof}

\begin{proposition}[Reduction to the unrefined topological recursion]\label{prop:Q=0RTR}
Given a genus zero degree two refined spectral curve satisfying \emph{Assumption {\rm \ref{ass:tr}}}, {\rm Definition~\ref{def:RTR}}  coincides with {\rm Definition~\ref{def:CEO} when $\mathscr{Q}=0$}.
\end{proposition}

\begin{proof}
It is straightforward from the recursion formula \eqref{eq:RTR} that when $\mathscr{Q}=0$, all $\omega_{g,n}$ vanish for $g\in\mathbb{Z}+\frac12$. Thus, we will only consider $\omega_{g,n+1}$ for all $g,n\in\mathbb{Z}_{\geq0}$ by induction on $2g+n\geq2$.

From \eqref{Rec0,3} together with the assumption, it is clear that ${\rm K}(p_0,p)\cdot{\rm Rec}_{0,3}^\mathscr{Q}(p,p_1,p_2)$ has poles only at $\mathcal{R}$, $J_0$, and $\sigma(J_0)$, so the recursion for $\omega_{0,3}$ becomes
\begin{equation}
    \omega_{0,3}(p_0,p_1,p_2)=-2\left(\sum_{r\in\mathcal{R}}\underset{p=r}{\text{Res}}+\sum_{r\in\sigma(\!J_0\!)}\underset{p=r}{\text{Res}}\right) {\rm K}(p_0,p)\!\cdot\! \text{Rec}^\mathscr{Q}_{0,3}(p,p_1,p_2).\label{w0,31}
\end{equation}
Since the integrand of \eqref{w0,31} is invariant under the involution $\sigma$, Lemma~\ref{lem:contour} implies that
\begin{align}
    \omega_{0,3}(p_0,p_1,p_2)&=-\sum_{r\in\mathcal{R}}\underset{p=r}{\text{Res}}\,{\rm K}(p_0,p)\cdot \text{Rec}^\mathscr{Q}_{0,3}(p,p_1,p_2) \notag \\
    &=-\sum_{r\in\mathcal{R}}\underset{p=r}{\text{Res}}\,{
    \rm K}(p_0,p)\cdot \frac12\Delta\omega_{0,2}(p,p_1)\cdot \Delta\omega_{0,2}(p,p_2).\label{Q=01}
\end{align}

Recall that $\omega_{0,2}(p,q)$ obeys
\begin{equation}
    \omega_{0,2}(p,q)+\omega_{0,2}(\sigma(p),q)=\frac{dx(p)dx(q)}{(x(p)-x(q))^2}.\label{w02base}
\end{equation}
The right hand side has a simple zero at $p\in\mathcal{R}$ whereas the recursion kernel ${\rm K}(p_0,p)$ has a simple pole there. Using this, one can rewrite \eqref{Q=01} in the following form:

\begin{equation}
    \omega_{0,3}(p_0,p_1,p_2)=\sum_{r\in\mathcal{R}}\underset{p=r}{\text{Res}}\,{\rm K}(p_0,p)\biggl({\omega}_{0,2}(p,p_1)\cdot{\omega}_{0,2}(\sigma(p),p_2)+{\omega}_{0,2}(\sigma(p),p_1)\cdot{\omega}_{0,2}(p,p_2)\biggr),
\end{equation}
which coincides with the unrefined topological recursion \eqref{TR}.

Likewise, $\text{Rec}^{\mathscr{Q}=0}_{1,1}(p)=-\omega_{0,2}(p,\sigma(p))$ is invariant under $\sigma$, so by a similar argument the recursion for $\omega_{1,1}$ becomes
\begin{align}
 \omega_{1,1}(p_0)=\sum_{r\in\mathcal{R}}\,\underset{p=r}{\text{Res}}\,{\rm K}(p_0,p) \cdot \omega_{0,2}(\sigma(p),p),
\end{align}
which again agrees with the unrefined recursion formula \eqref{TR} for $\omega_{1,1}$. 

As a consequence, we have the following property of the unrefined topological recursion obtained in \cite[Theorem 4.4]{EO2}:
\begin{equation}
    \omega_{0,3}(p_0,p_1,p_2)+\omega_{0,3}(\sigma(p_0),p_1,p_2)=0,\;\;\;\;\omega_{1,1}(p_0)+\omega_{1,1}(\sigma(p_0))=0.
\end{equation}

Now, we assume that up to $m=2g+n\geq2$, all $\omega_{g,n+1}(p_0,J)$ agree with those from the unrefined topological recursion. As above this implies that  up to $2g+n=m\geq2$ we have
\begin{equation}
    \omega_{g,n+1}(p_0,J)+\omega_{g,n+1}(\sigma(p_0),J)=0,\label{antiinvolution}
\end{equation}
due to \cite[Theorem 4.4]{EO2}. One can then check that for all $g,n$ with $2g+n=m+1$,  $\text{Rec}^{\mathscr{Q}=0}_{g,n+1}(p,J)$ is invariant under the involution $\sigma$. Therefore, by applying Lemma~\ref{lem:contour}, one finds that the recursive formula for $\omega_{g,n+1}$ with $2g+n=m+1$ coincides with that of the unrefined topological recursion, similar to the case of $\omega_{0,3}$ and $\omega_{1,1}$.
\end{proof}

\begin{remark}\label{rem:0n}
It is immediately clear from the recursion formula \eqref{eq:RTR} that the $\omega_{g,n+1}$ depend polynomially on $\mathscr{Q}$. In particular, the $\mathscr{Q}$-independent terms $\omega_{g,n+1}\big|_{\mathscr{Q}=0}$ are recursively computed only by $\mathscr{Q}$-independent terms in $\text{Rec}^\mathscr{Q}_{g,n+1}$. Proposition~\ref{prop:Q=0RTR} implies that $\omega_{g,n+1}\big|_{\mathscr{Q}=0}$ from the refined topological recursion matches with $\omega_{g,n+1}$ from the unrefined topological recursion for each $g$ and $n$. In other words, the recursion and the limit $\mathscr{Q}=0$ commute.
\end{remark}



\subsection{Properties of refined topological recursion}\label{sec:properties}
In this section, we will prove key properties of the refined topological recursion analogous to Theorem~\ref{thm:TR}. Since the pole structure of $\omega_{g,n+1}(p_0,J)$ is quite different in the refined setting, it requires a more delicate treatment than one may naively imagine.

We have:
\begin{theorem}\label{thm:RTR}
Let $\mathcal{S}^{\bm \mu }$ be a genus zero degree two refined spectral curve satisfying {\rm Assumption \ref{ass:tr}}. For any $(g,n)$ with $2g,n\in\mathbb{Z}_{\geq0}$ and $2g+n\geq2$, the multidifferentials $\omega_{g,n+1}$ constructed from the refined topological recursion satisfy the following properties:
\begin{enumerate}[label={\bf RTR\arabic*:}, ref=\arabic*,leftmargin=1.7cm]
\itemsep2pt 
\item\label{RTR1} The multidifferential $\omega_{g,n+1}$ is symmetric;

\item\label{RTR2} All poles of $\omega_{g,n+1}$ in the first variable lie in $\mathcal{R}^*\cup\sigma(J)$;

\item\label{RTR3} At any $ o\in \Sigma$, $\omega_{g,{n+1}}$ is residue-free in the first (thus, any) variable:
\begin{equation}
    \underset{p=o}{{\rm Res}}\,\omega_{g,n+1}(p,J)=0;
\end{equation}

\item \label{RTR4}  For $n>0$ we have
\begin{equation}
(2g-2+n)\,\omega_{g,n}(J)=-\left(\sum_{r\in\mathcal{R}^*}\underset{p=r}{{\rm Res}}+\sum_{r\in\sigma(J)}\underset{p=r}{{\rm Res}}\right)\Phi(p)\cdot\omega_{g,n+1}(p,J).\label{Rinverse}
\end{equation}

\end{enumerate}
where $J=(p_1,\ldots,p_n)\in \Sigma^n$, and we assume $p_j\neq \sigma(p_k)$ for all $i,j=1,\ldots,n$.
\end{theorem}

\subsubsection{Proof of RTR\ref{RTR1}}

We will use several lemmas in the proof of RTR\ref{RTR1}, whose proofs we defer to Appendix~\ref{sec:proofs}:
\begin{lemma}[Structure of $\omega_{\frac{1}{2},1}$]\label{lem:pole1/2,1} All poles of
$\omega_{\frac12,1}$ are simple and lie in $\mathcal{R}\cup\widetilde{\mathcal{P}}$. Explicitly,
let ${\rm ord}_p(y)$ be the order of the zero of $y$ at $p\in\widetilde{\mathcal{P}}$ and $-{\rm ord}_p(y)$ be the order of the pole, then $\omega_{\frac12,1}$ is given by
\begin{equation}
    \omega_{\frac12,1}(p_0)=\frac{\mathscr{Q}}{2}\left(-\frac{dy(p_0)}{y(p_0)}+\sum_{p\in \widetilde{\mathcal{P}}_+}\,\mu_p\,{\rm ord}_p(y)\,\eta_p(p_0)\right),\label{w1/2,1-2}
\end{equation}
where $\eta_p$ is the unique meromorphic differential on $\mathbb{P}^1$ with residue $+1$ at $p$ and $-1$ at $\sigma(p)$.
\end{lemma}
\begin{lemma}[Pole structure in the first variable]\label{lem:poles}
For $2g+n\geq2$, all poles of $\omega_{g,n+1}(p_0,J)$, with respect to $p_0$, lie in $\mathcal{R}^*\cup \sigma(J)$.
\end{lemma}

\begin{remark}
At an ineffective ramification point, the kernel ${\rm K}(p_0,p)$ has at least a triple zero in $p$, $\omega_{\frac12,1}$ has at most simple pole, and thanks to Lemma \ref{lem:poles} all other $\omega_{g,n+1}$ are regular in the first variable. Thus, it is clear that in the recursion formula there will be no contributions from the residues at these points, and henceforth we freely remove them when convenient.
\end{remark}

\begin{lemma}[Loop equations]\label{lem:GLP}
Define a quadratic (in $p$) differential $P_{g,n+1}(p,J)$ for $2g+n\geq2$ by
\begin{equation}
    P_{g,n+1}(p,J):=2\omega_{0,1}(p)\omega_{g,n+1}(p,J)+{\rm Rec}_{g,n+1}^{\mathscr{Q}}(p,J).\label{GLP}
\end{equation}
Then:
\begin{enumerate}[label=\roman*)]
    \item $P_{g,n+1}$ is invariant under the involution $\sigma$,
    \item $P_{g,n+1}$ is holomorphic at every effective ramification point $r\in\mathcal{R}^*$.
\end{enumerate}
\end{lemma}
It turns out that the proof of the symmetry for $\omega_{\frac12,2}$ is quite different from that for other $\omega_{g,n+1}$, hence we will also defer its proof to Appendix~\ref{sec:proofs}:
\begin{lemma}\label{lem:w12,2} The bidifferential $\omega_{\frac{1}{2},2}$ is symmetric:
\begin{equation}
    \omega_{\frac12,2}(p_0,p_1)=\omega_{\frac12,2}(p_1,p_0).
\end{equation}
\end{lemma}



Equipped with these lemmas, we will now prove the symmetry for the other differentials.

\begin{proposition}[RTR\ref{RTR1}]\label{thm:sym}
For any $2g+n\geq2$, the multidifferential $\omega_{g,n+1}$ is symmetric.
\end{proposition}
\begin{proof}
We will proceed by induction on $m:=2g+n$. 

Since it is clear from the recursion formula \eqref{RecQ} that $\omega_{g,n+2}(p_0,q_0,J)$ is symmetric under the permutation of the last  $n+1$ entries, it is sufficient to show symmetry in the first two variables:
\begin{equation}
    \omega_{g,n+2}(p_0,q_0,J)=\omega_{g,n+2}(q_0,p_0,J).
\end{equation}

From Remark \ref{rem:0n} we have that $\omega_{0,3}$ is symmetric.  There is nothing to prove for $\omega_{1,1}$, and Lemma~\ref{lem:w12,2} gives that $\omega_{\frac12,2}$ is also symmetric. Let us then assume that $\omega_{g,n+2}(p_0,q_0,J)$ are symmetric under permutation of arguments for all $g,n$, up to $2g+n=m\geq1$. Then, Lemma~\ref{lem:poles} holds not only with respect to $p_0$, but also with respect to the other arguments $q_0,J$. For all $g,n$ with $2g+n=m+1$, the recursion formula \eqref{eq:RTR} simply reads
\allowdisplaybreaks[0]
\begin{align}
    \omega_{g,n+2}(p_0,q_0,J)=&-2\left(\sum_{r\in\mathcal{R}^*}\underset{p=r}{\text{Res}}+\!\!\!\!\sum_{r\in\sigma(J_0)}\underset{p=r}{\text{Res}}+\underset{p=\sigma(q_0)}{\text{Res}}\right){\rm K}(p_0,p) \nonumber \\& \hspace{0.25cm}\times\Biggl(\Delta\omega_{0,2}(p,q_0)\cdot\omega_{g,n+1}(p,J)
    +\sum_{i=1}^n\Delta\omega_{0,2}(p,p_i)\cdot\omega_{g,n+1}(p,q_0,\widehat{J}_i)\nonumber\\
&\hspace{1.5cm}+2\sum^{**}_{\substack{g_1+g_2=g \\ J_1\sqcup J_2=J}} \omega_{g_1,n_1+2}(p,q_0,J_1)\cdot \omega_{g_2,n_2+1}(p,J_2) \nonumber\\
&\hspace{1.5cm}+\omega_{g-1,n+3}(p,p,q_0,J)+\mathscr{Q}\,dx(p)\cdot d_p\left(\frac{ \omega_{g-\frac12,n+2}(p,q_0,J)}{dx(p)}\right)\Biggr).\label{sym1}
\end{align}
\allowdisplaybreaks[1]
The integrand has no poles at $\widetilde{\mathcal{P}}_+$ and thus we have removed it from the summation. This follows by noticing that by the inductive assumption and Lemma \ref{lem:poles}, all $\omega_{g,n}$ appearing on the right hand side except $(g,n)=(\frac{1}{2},1)$ are regular at all $p\in\widetilde{\mathcal{P}}_+$, so the only possible poles might appear from terms in the second sum containing a single factor of $\omega_{\frac{1}{2},1}$. By Lemma \ref{lem:pole1/2,1}, the only possible poles at $\widetilde{\mathcal{P}}_+$ are simple, but since $\widetilde{\mathcal{P}}_+\subset \mathcal{P}_+$ are poles of $ y dx$, the kernel ${\rm K}(p_0,p)$ must have a zero there, so that the integrand is regular.

Since the remaining procedure is heavily computational, let us first describe the strategy:
\begin{enumerate}
    \item Since $\omega_{g,n+2}$ are symmetric up to $2g+n=m$ by the inductive assumption, we can move $q_0$ to the first argument everywhere it appears in the integrand. Then we can re-apply the recursion with respect to $q_0$ in the integrand. That is, whenever $\omega_{\tilde{g},\tilde{n}+1}(q_0,\,\cdot\,)$ appears for any $\tilde{g},\tilde{n}$ with $2\tilde{g}+\tilde{n}\geq2$, we replace it with the right hand side of the recursion formula \eqref{eq:RTR}. Note that the first term in the second line in \eqref{sym1} remains unchanged since $\omega_{0,2}(q_0,p)$ is an initial datum.
    \item As a consequence, the whole expression will be a sum of terms each consisting of two residue computations, one appearing in \eqref{sym1}, and one coming from the recursion for each $\omega_{\tilde{g},\tilde{n}+1}(q_0,\,\cdot\,)$ (all of whose dummy variables we will call $q$, since no product of such terms appears):
    \begin{equation}\label{eq:doublesum}
         \left(\sum_{r\in\mathcal{R}^*}\underset{p=r}{\text{Res}}+\!\!\!\!\sum_{r\in\sigma(J)}\underset{p=r}{\text{Res}}+\underset{p=\sigma(p_0)}{\text{Res}}+\underset{p=\sigma(q_0)}{\text{Res}}\right)\left(\sum_{s\in\mathcal{R}^*}\underset{q=s}{\text{Res}}+\!\!\!\!\sum_{s\in\sigma(J)}\underset{q=s}{\text{Res}}+\underset{q=\sigma(p)}{\text{Res}}+\underset{q=\sigma(q_0)}{\text{Res}}\right).
    \end{equation}
    \item When one exchanges the order of the two residues (c.f. \cite[proof of Theorem 4.6]{EO}), the integrand in the double recursion shows that \eqref{eq:doublesum} becomes
    \begin{equation}\label{eq:doublesum2}
        \left(\sum_{s\in\mathcal{R}^*}\underset{q=s}{\text{Res}}+\sum_{s\in\sigma(J)}\underset{q=s}{\text{Res}}+\underset{q=\sigma(p_0)}{\text{Res}}+\underset{q=\sigma(q_0)}{\text{Res}}\right)\left(\underset{p=q}{\text{Res}}+\sum_{r\in\mathcal{R}^*}\underset{p=r}{\text{Res}}+\!\!\!\!\sum_{r\in\sigma(J)}\underset{p=r}{\text{Res}}+\underset{p=\sigma(q)}{\text{Res}}+\underset{p=\sigma(p_0)}{\text{Res}}\right).
    \end{equation}
    It is clear the only difference between these two expressions is the term arising from the extra ``$\underset{p=q}{\text{Res}}$'' appearing in the second parentheses, so we focus only on such terms. \item By the inductive assumption, Lemma \ref{lem:poles} implies that the only poles contributing to ``$\underset{p=q}{\text{Res}}$'' are due to $\omega_{0,2}(p,q)$ appearing in a term of the integrand. One then checks that after taking $\underset{p=q}{\text{Res}}$, some terms will become manifestly symmetric in $p_0\leftrightarrow q_0$. The collection of remaining terms will become invariant under the involution, so that the second set of residues will be reduced to ramification points $\mathcal{R}^*$ by Lemma~\ref{lem:contour}. 
    \item Furthermore, one can show that the collection of the remaining terms coincides $P_{g,n+1}$, hence it is holomorphic at $\mathcal{R}^*$ by the loop equation (Lemma~\ref{lem:GLP}). Therefore, it vanishes after taking the residues at $\mathcal{R}^*$.
\end{enumerate}

We omit writing down most of terms in the double-recursion as they are straightforward and tedious, but after following the above steps, seemingly asymmetric terms in $p_0\leftrightarrow q_0$ are collected into the following form:
\begin{align}
    &\omega_{g,n+2}(p_0,q_0,J)\nonumber\\
    &\hspace{0.0cm}=(\text{sym}_1)+\left(\sum_{r\in\mathcal{R}^*}\underset{p=r}{\text{Res}}+\!\!\!\!\sum_{r\in\sigma(J)}\underset{p=r}{\text{Res}}+\underset{p=\sigma(p_0)}{\text{Res}}+\underset{p=\sigma(q_0)}{\text{Res}}\right){\rm K}(p_0,p)\cdot\Delta\omega_{0,2}(p,q_0)\cdot\omega_{g,n+1}(p,J)\nonumber\\
    &\hspace{1.55cm}+\left(\sum_{s\in\mathcal{R}^*}\underset{q=s}{\text{Res}}+\!\!\!\!\sum_{s\in\sigma(J)}\underset{q=s}{\text{Res}}+\underset{q=\sigma(p_0)}{\text{Res}}+\underset{q=\sigma(q_0)}{\text{Res}}\right)\underset{p=q}{\rm Res}\, {\rm K}(p_0,p)\cdot {\rm K}(q_0,q)\nonumber\\
    &\hspace{2cm}\times \Bigg(\omega_{0,2}(p,q)\cdot\Bigg(2\sum_{\substack{g_1+g_2=g\\J_1\sqcup J_2=J}}\omega_{g_1,n_1+1}(p,J_1)\cdot\omega_{g_2,n_2+1}(q,J_2)+\omega_{g-1,n+2}(p,q,J)\Bigg)\nonumber\\
    &\hspace{4cm}+\mathscr{Q}\,dx(p)\cdot\omega_{g-\frac12,n+1}(p,J)\cdot d_p\left(\frac{\omega_{0,2}(p,q)}{dx(p)}\right)\Bigg).\label{sym2}
\end{align}
where $(\text{sym}_1)$ denotes terms manifestly symmetric in $p_0\leftrightarrow q_0$. We then evaluate the residue at $p=q$ and integrate by parts to find the last three lines of \eqref{sym2} give
\begin{align}
   \left(\sum_{s\in\mathcal{R}^*}\underset{q=s}{\text{Res}}+\!\!\!\!\sum_{s\in\sigma(J)}\underset{q=s}{\text{Res}}+\underset{q=\sigma(p_0)}{\text{Res}}+\underset{q=\sigma(q_0)}{\text{Res}}\right){\rm K}(p_0,q) \cdot\Delta\omega_{0,2}(q,q_0)\cdot\frac{\text{Rec}_{g,n+1}(q,J)}{2\omega_{0,1}(q)}+(\text{sym}_2),\label{sym3}
\end{align}
where $(\text{sym}_2)$ again denotes manifestly symmetric terms in $p_0\leftrightarrow q_0$, and we have omitted tedious computational manipulations required to obtain \eqref{sym3} from \eqref{sym2}. Together with the first term in \eqref{sym2}, we obtain
 \begin{equation}
     \omega_{g,n+2}(p_0,q_0,J)= \sum_{s\in\mathcal{R}^*}\underset{q=s}{\text{Res}}\,{\rm K}(p_0,q) \cdot\Delta\omega_{0,2}(q,q_0)\cdot\frac{P_{g,n+1}(q,J)}{2\omega_{0,1}(q)}+(\text{sym}_1)+(\text{sym}_2),\label{sym4}
 \end{equation}
where we used Lemma~\ref{lem:GLP} and Lemma~\ref{lem:contour} to simplify the expression. Finally, Lemma~\ref{lem:GLP} and \eqref{Pg,n+1-1} imply that the first term in \eqref{sym4} vanishes by \cite[Lemma A.1]{EO2}, giving the result.
\end{proof}

\subsubsection{Proof of RTR\ref{RTR2}, RTR\ref{RTR3}, and RTR\ref{RTR4}}

Having established the symmetry of $\omega_{g,n+1}$, we now prove the remaining properties.

\begin{corollary}[RTR\ref{RTR2}]\label{coro:poles}
For all $g,n$ with $2g+n\geq2$, all poles of $\omega_{g,n+1}(p_0,J)$ (in any variable) lie in $\mathcal{R}^*\cup\sigma(J_0)$;
\end{corollary}
\begin{proof}
Immediate from Theorem~\ref{thm:sym} and Lemma~\ref{lem:poles}.
\end{proof}

The proofs of the other two properties are intertwined: we will first partially prove RTR\ref{RTR3} in order to prove RTR\ref{RTR4}, and then use RTR\ref{RTR4} to complete the proof.
\begin{proposition}\label{lem:residue1}
For all $2g,n\in\mathbb{Z}_{\geq0}$, $\omega_{g,{n+2}}$ is residue-free in the first (thus, any) variable:
\begin{equation}
    \underset{p=o}{{\rm Res}}\,\omega_{g,n+2}(p,p_0,J)=0;
\end{equation}
at any $o \in \Sigma$.
\end{proposition}
\begin{proof}
We proceed by induction on $2g+n\geq1$. By Remark \ref{rem:0n} all $\omega_{0,n}$ agree with their unrefined versions, and thus have no residues (TR\ref{TR3}). Let us next consider $\omega_{\frac12,2}$. It follows from the recursion formula \eqref{eq:RTR} that for any $o\in\Sigma$:
\begin{align}
&\underset{p_1=o}{\text{Res}}\,\omega_{\frac12,2}(p_0,p_1)\nonumber\\
    &=-2\underset{p_1=o}{\text{Res}}\left(\sum_{r\in\mathcal{R}}\underset{p=r}{\text{Res}}+\underset{p=\sigma(p_0)}{\text{Res}}+\underset{p=\sigma(p_1)}{\text{Res}}\right){{\rm K}(p_0,p)}\nonumber\\
    &\hspace{3cm}\times \left(\Delta\omega_{0,2}(p,p_1)\cdot\omega_{\frac12,1}(p)-\mathscr{Q}\,dx(p)\cdot d_{p}\left(\frac{ \omega_{0,2}(\sigma(p),p_1)}{dx(p)}\right)\right)\nonumber\\
    &=-2\left(\sum_{r\in\mathcal{R}}\underset{p=r}{\text{Res}}+\underset{p=\sigma(p_0)}{\text{Res}}+\underset{p=\sigma(p_1)}{\text{Res}}\right)\left(\underset{p_1=o}{\text{Res}}+\underset{p_1\in \tilde P}{\text{Res}}\right){{\rm K}(p_0,p)}\nonumber\\
    &\hspace{3cm}\times\left(\Delta\omega_{0,2}(p,p_1)\cdot\omega_{\frac12,1}(p)-\mathscr{Q}\,dx(p)\cdot d_{p}\left(\frac{ \omega_{0,2}(\sigma(p),p_1)}{dx(p)}\right)\right)\nonumber\\
    &=0,
\end{align}
where $\tilde P$ is a collection of points that may appear when one exchanges the order of residues, depending on where $o$ is. The last line follows since the $p_1$-residue on the second line vanishes. To see this, note that there is no $\omega_{\frac12,1}(p_1)$ in the recursion: what appears is $\omega_{\frac12,1}(p)$, hence $\omega_{\frac12,2}(p_0,p_1)$ has no residue even though $\omega_{\frac12,1}(p_0)$ does. 

For other $\omega_{g,n+2}$ we can prove the statement by induction. The trick is the same, and it is easy to show that exchanging the order of residues does not give any contribution by the same reason as above. Alternatively, one can take the expression \eqref{pole2} and compute the residue with respect to the second variable.
\end{proof}

We now prove RTR\ref{RTR4}. 

\begin{proposition}[RTR\ref{RTR4}]\label{thm:Rinverse}
For any $2g+n\geq2$ with $n>0$, we have
\begin{equation}
(2g-2+n)\,\omega_{g,n}(J)=-\left(\sum_{r\in\mathcal{R}^*}\underset{p=r}{{\rm Res}}+\sum_{r\in\sigma(J)}\underset{p=r}{{\rm Res}}\right)\Phi(p)\cdot\omega_{g,n+1}(p,J)\label{Rinverse11},
\end{equation}
where $\mathcal{R}^*$ is the set of effective ramification points, and $\Phi$ is any primitive of $ydx$.
\end{proposition}


\begin{proof} 
We note first that the right hand side of \eqref{Rinverse11} is indeed well-defined by  Proposition~\ref{lem:residue1}. Note also that Remark \ref{rem:0n} and TR4 implies that \eqref{Rinverse11} holds between arbitrary $\omega_{0,n}$ and $\omega_{0,n+1}$ with $n>1$. In particular, the result will hold for $2g+n=2$ once we check it holds between $\omega_{\frac{1}{2},1}$ and $\omega_{\frac{1}{2},2}$.

In the proof of Lemma \ref{lem:w12,2} it is shown that $\omega_{\frac{1}{2},{2}}$ can be massaged into the form 
\begin{align}
    \omega_{\frac12,2}(p_0,p_1)=\mathscr{Q}\left(\sum_{r\in\mathcal{R}^*}\underset{p=r}{\text{Res}}+\!\!\underset{p=\sigma(p_0)}{\text{Res}}\!+\!\!\underset{p=\sigma(p_1)}{\text{Res}}\right)&\Biggl(\frac{\Delta\omega_{0,2}(p,p_0)\cdot\Delta\omega_{0,2}(p,p_1)}{4\, y(p)dx(p)}\nonumber\\
    &-{\rm K}(p_0,p)\!\cdot\!\Delta\omega_{0,2}(p,p_1)\!\cdot\!\!\!\sum_{q\in \widetilde{\mathcal{P}}_+}\,\mu_q\,{\rm ord}_q(y)\eta_q(p)\Biggr).\label{w1/2,22again}
\end{align}
It is also shown that the contribution to $\omega_{\frac{1}{2},2}$ from evaluating the residues of the second term in the integrand above has a double pole at $p_0=r\in\mathcal{R}^*$ when $r$ is a double zero of $\omega_{0,1}$ and it vanishes otherwise. In the former case, we note that the expansion in a coordinate $z$ centred at $r\in\mathcal{R}^*$ of the primitive $\Phi(z)=\Phi(r)+c z^3+\ldots$ has a constant term and third order term, but no linear term, thus, we find that the second line of \eqref{w1/2,22again} has no contribution to \eqref{Rinverse11}. Thus, we have:

\begin{align}
    &\left(\sum_{r\in\mathcal{R}^*}\underset{p=r}{\text{Res}}+\underset{p=\sigma(p_0)}{\text{Res}}\right)\Phi(p)\cdot \omega_{\frac12,2}(p,p_0)\nonumber\\
    &=\mathscr{Q}\left(\sum_{r\in\mathcal{R}^*}\underset{p=r}{\text{Res}}+\underset{p=\sigma(p_0)}{\text{Res}}\right)\left(\sum_{s\in\mathcal{R}^*}\underset{q=s}{\text{Res}}+\underset{q=\sigma(p_0)}{\text{Res}}+\underset{q=\sigma(p)}{\text{Res}}\right)\Phi(p)\,\frac{\Delta\omega_{0,2}(q,p_0)\cdot\Delta\omega_{0,2}(q,p)}{4\, y(q)dx(q)},\label{Rinverse2}
\end{align}
where we used the symmetry of $\omega_{\frac{1}{2},2}$ to apply the recursion in $p_0$. If one exchanges the order of residues, we have
\begin{equation}
\mathscr{Q}\left(\sum_{s\in\mathcal{R}^*}\underset{q=s}{\text{Res}}+\underset{q=\sigma(p_0)}{\text{Res}}\right)\left(\underset{p=q}{{\rm Res}}+\sum_{r\in\mathcal{R}^*}\underset{p=r}{\text{Res}}+\underset{p=\sigma(p_0)}{\text{Res}}+\underset{p=\sigma(q)}{\text{Res}}\right)\Phi(p)\,\frac{\Delta\omega_{0,2}(q,p_0)\cdot\Delta\omega_{0,2}(q,p)}{4\, y(q)dx(q)}\label{Rinverse2b}.
\end{equation}
Then since only $\omega_{0,2}(p,q)$ and $\omega_{0,2}(p,\sigma(q))$ appears in the integrand of \eqref{Rinverse2b}, evaluating the residues gives zero, and thus
\begin{equation}
    \left(\sum_{r\in\mathcal{R}^*}\underset{p=r}{\text{Res}}+\underset{p=\sigma(p_0)}{\text{Res}}\right)\Phi(p)\cdot  \omega_{\frac12,2}(p,p_0)=0.
\end{equation}

We now proceed by induction on $2g+n\geq2$. First, by applying the refined topological recursion, it follows
\begin{align}
&\left(\sum_{r\in\mathcal{R}^*}\underset{p=r}{\text{Res}}+\!\!\!\!\sum_{s\in\sigma(J)}\underset{p=r}{\text{Res}}\right)\Phi(p)\cdot  \omega_{g,n+1}(p,J)\nonumber\\
&=-2\left(\sum_{s\in\mathcal{R}^*}\underset{q=s}{\text{Res}}+\!\!\!\!\sum_{s\in\sigma(J)}\underset{q=s}{\text{Res}}\right)\left( \underset{p=q}{\text{Res}}+\underset{p=\sigma(q)}{\text{Res}}+\sum_{r\in\mathcal{R}^*}\underset{p=r}{\text{Res}}+\!\!\!\!\sum_{r\in\sigma(J)}\underset{p=r}{\text{Res}}\right)\nonumber\\&\hspace{80mm}\times
{\rm K}(p_1,q)\cdot\Phi(p)\cdot\text{Rec}_{g,n+1}^\mathscr{Q}(q,p,\widehat{J}_1),\label{Rinverse3}
 \end{align}
where we have already exchanged the order of contour integrals as we did in the proof of Proposition \ref{thm:sym}. 

We evaluate this by breaking the integrand up into two pieces. We first consider all terms in the integrand that do not contain a factor of $\omega_{0,2}(q,p)$. To evaluate the inner residues, we note such terms have no residues at $p=q$ thanks to Lemma~\ref{lem:poles}, so the inductive assumption gives:
\begin{align}
&\left( \underset{p=q}{\text{Res}}+\underset{p=\sigma(q)}{\text{Res}}+\sum_{r\in\mathcal{R}^*}\underset{p=r}{\text{Res}}+\!\!\!\!\sum_{r\in\sigma(J)}\underset{p=r}{\text{Res}}\right){\Phi(p)}\cdot \text{Rec}^{\mathscr{Q},**}_{g,n+1}(q,p,\widehat{J}_1)=-(2g+n-3)\cdot \text{Rec}^{\mathscr{Q}}_{g,n}(q,\widehat{J}_1),
\end{align}
where we wrote $\text{Rec}^{\mathscr{Q},**}_{g,n+1}$ to denote only the terms of $\text{Rec}^{\mathscr{Q}}_{g,n+1}$ independent of $\omega_{0,2}$.\footnote{Note that the computation is in part parallel to \cite[Eq. (A-31)]{EO2}, which the reader may refer to for details.} Multiplying by $-2\,{\rm K}(p_1,q)$ and applying the outer residues to the right hand side, we immediately obtain $-(2g+n-3)\,\omega_{g,n}(J)$ from the recursion formula \eqref{eq:RTR}.

It remains to evaluate the terms involving $\omega_{0,2}(q,p)$:
\begin{align} \label{Rinverse5}
\left( \underset{p=q}{\text{Res}}+\underset{p=\sigma(q)}{\text{Res}}+\sum_{r\in\mathcal{R}^*}\underset{p=r}{\text{Res}}+\!\!\!\!\sum_{r\in\sigma(J)}\underset{p=r}{\text{Res}}\right)\Phi(p)\cdot \Delta\omega_{0,2}(q,p)\cdot\omega_{g,n}(q,\widehat{J}_1)&\nonumber\\=\qquad 2 y(q) dx(q)\cdot \omega_{g,n}(q,\widehat{J}_1)&.
\end{align}
where the latter two residues vanish trivially, and the first two follow from 
\begin{align}
    \underset{p=q}{\rm Res}\,\Phi(p)\omega_{0,2}(p,q)=\omega_{0,1}(q), \qquad \underset{p=\sigma(q)}{\rm Res}\Phi(p)\omega_{0,2}(p,\sigma(q))=\omega_{0,1}(\sigma(q)).
\end{align}
Noting that $2 y(q) dx(q)$ is the denominator of the kernel, multiplying \eqref{Rinverse5} by $-2{\rm K }(p_1,q)$ and taking the residues in $q$, we must evaluate
\begin{equation}\label{equation255}
   -\left(\sum_{s\in\mathcal{R}^*}\underset{q=s}{\text{Res}}+\!\!\!\!\sum_{r\in\sigma(J)}\underset{q=s}{\text{Res}}\right) \int_{\sigma(q)}^q\omega_{0,2}(p_1,\cdot)\cdot \omega_{g,n}(q,\widehat{J}_1).
\end{equation}
The integrand can have poles in $q$ only at $\mathcal{R}^*$, $\sigma({J})$ or $p_1$ thanks to RTR\ref{RTR2}. Then since the sum of residues of a meromorphic form vanishes, \eqref{equation255} becomes
\begin{equation}
\underset{q=p_1}{\text{Res}}\int_{\sigma(q)}^q\omega_{0,2}(p_1,\cdot)\cdot \omega_{g,n}(q,\widehat{J}_1) = -\omega_{g,n}(J).
\end{equation}
Putting the two together, we obtain the result.
\end{proof}


Having established RTR\ref{RTR4}, we may now complete the proof of RTR\ref{RTR3}:
\begin{proposition}[RTR\ref{RTR3}]\label{prop:residue2}
For $2g+n\geq2$, and any $o\in\Sigma$ $\omega_{g,n+1}$ is residue-free in the first (thus, any) variable:
\begin{equation}
\underset{p=o}{{\rm Res}}\,\omega_{g,n+1}(p,J)=0.\label{residue2}
\end{equation}
\end{proposition}

\begin{proof}
It remains only to show the case of $\omega_{g,1}$, $g\geq1$. Thanks to RTR\ref{RTR4}, \eqref{Rinverse} shows that
\begin{align}
&(2g-1)\,\omega_{g,1}(p_0)\\
&=\left(\sum_{r\in\mathcal{R}^*}\underset{p=r}{\text{Res}}+\underset{p=\sigma(p_0)}{\text{Res}}\right){\Phi(p)}\, \omega_{g,2}(p,p_0)\nonumber\\
&=-2\left(\sum_{r\in\mathcal{R}}\underset{p=r}{\text{Res}}+\underset{p=\sigma(p_0)}{\text{Res}}\right)\left(\sum_{s\in\mathcal{R}}\underset{q=s}{\text{Res}}+\underset{q=\sigma(p)}{\text{Res}}+\underset{q=\sigma(p_0)}{\text{Res}}\right){\Phi(p)}\cdot{{\rm K}(p,q)}\cdot\text{Rec}_{g,2}^\mathscr{Q}(q,p_0), 
\end{align}
We now take the residue in $p_0$ and observe that $\underset{{p_0=o}}{\text{Res}}$ commutes with both of the other residue terms. This follows since the precise form of ${\rm Rec}^{\mathscr{Q}}_{g,2}(q,p_0)$ is a sum of $p_0$-independent coefficients times $\omega_{\tilde{g},\tilde{n}}$'s with $p_0$ in one argument, which one can check all have $\tilde{n}>1$. Hence the proposition follows from the residue-freeness already shown.
\end{proof}

\section{Quantum curves}\label{sec:QC}
Having established the necessary definitions and properties of the refined recursion, we now turn to the task of constructing the corresponding quantum curves. 

Throughout this section we fix the standard coordinate $z$ on $\mathbb{P}^1$, and write all expressions in this coordinate for convenience (occasionally blurring the distinction between point and coordinate). Furthermore, since there is a canonical choice of bidifferential $B$ on $\mathbb{P}^1$, we will always assume
\begin{equation}\label{eq:canonicalB}
    B(z_1,z_2)=\frac{dz_1dz_2}{(z_1-z_2)^2}.
\end{equation}
throughout this section.

The results of this section are a generalization of the results of \cite{IKT1,IKT2}, and we follow their presentation closely.

\subsection{Definition of quantum curve}

Let $\mathcal{S}^{\bm \mu}=(\Sigma,x,y,B,D({\bm \mu}))$ be a genus zero degree two refined spectral curve satisfying Assumption \ref{ass:tr}. From the refined topological recursion one may construct a so-called wavefunction. In fact, it turns out we may write an entire family of wavefunctions --- for this purpose, we take a (complex) divisor
\begin{equation}
D(z; {\bm \nu}) := [z] - \sum_{p \in {\mathcal P}} \nu_p [p],
\end{equation}
on ${\Sigma}$ which depends on $z \in {\Sigma}$ and a tuple ${\bm \nu} = (\nu_p)_{p \in \mathcal{P}}$ of complex parameters, which we call {\em quantization parameters}, satisfying 
\begin{equation} \label{eq:relation-nu}
\sum_{p \in {\mathcal P}} \nu_p = 1. 
\end{equation}
\noindent Integration of a residue-free meromorphic one-form $\omega$ with respect to the divisor $D(z;{\bm \nu})$ will by definition mean
\begin{equation}
\int_{D(z;{\bm \nu})}\omega:= \sum_{p\in \mathcal{P}}\nu_p \, \int_{p}^z\omega,
\end{equation}
which is well-defined.

In order to define the wavefunction, we integrate $\omega_{g,n}$ along a divisor repeatedly. The following lemma assures good behaviour of such multi-integrals:

\begin{lemma}\label{lem:residue3}
Given a genus zero degree two refined spectral curve, let $p_i \in\Sigma$ for $i=0,\ldots n$ and suppose that $p_i\not\in\mathcal{R}^*\cup\mathcal{P}$, and $p_i\neq \sigma(p_j)$ for any $i,j$. Suppose furthermore that lower endpoints $q_i\in {\mathcal{P}}$ are chosen. Then for all $2g+n\geq2$, the integral
\begin{equation}
   \int^{p_0}_{q_0}\cdots\int^{p_{n}}_{q_{n}}\omega_{g,n+1}(\zeta_0,\zeta_1,\ldots,\zeta_n)\label{residue3}
\end{equation}
where the $(i+1)th$ integral from the left is in the $\zeta_i$ variable, exists and defines a meromorphic function on $\Sigma$ with respect to each $p_i$, regular at $p_i=p_j$ for any $i,j$, and is independent of the path of integration. 
\end{lemma}
\noindent We defer the proof to Appendix \ref{sec:proofs}.

\begin{remark}\label{rem:continuous}
    Note that this integral is in fact not guaranteed to be continuous in $p_i$ at the excluded points, even if it exists. Thus, we have said that \eqref{residue3} \emph{defines} a meromorphic function, with the understanding that it is given by a rational expression whose value is given by the integral \eqref{residue3} on a dense open subset of $\Sigma^{n+1}$.
\end{remark}

We may now define the \emph{(refined) topological recursion wavefunction} by (see also \cite{MS,CHMS} for a special choice of ${\bm \nu}$ and ${\bm \mu}$):

\begin{align}
&    \varphi^{\rm \mathsmaller{TR}}(z) = \varphi^{\rm \mathsmaller{TR}}(z;{\bm \mu}, {\bm \nu},\beta):= \nonumber \\
&    \exp \left( \sum_{g\in \frac{1}{2}\mathbb{Z}_{\geq 0},n>0} \dfrac{\hbar^{2g-2+n}}{\beta^{n/2}} \dfrac{1}{n!} \underset{\zeta_1 \in D(z;{\bm \nu})}{\int}\!\!\!\ldots\!\!\! \underset{\zeta_n \in D(z;{\bm \nu})}{\int}\omega_{g,n}(\zeta_1,\ldots,\zeta_n)-\delta_{g,0}\delta_{n,2}\frac{dx(\zeta_1)dx(\zeta_2)}{(x(\zeta_1)-x(\zeta_2))^2}\right),
\end{align}
which should be understood as (the exponential of) a formal series in $\hbar$ with coefficients depending on $z$ and parameters $\beta$, ${\bm \mu}$, and ${\bm \nu}$. Note that when $\beta = 1$ this indeed reproduces the definition of the topological recursion wavefunction in the unrefined case, and that the integration along the divisor with upper endpoints all specialized to $z$ makes sense thanks to Lemma \ref{lem:residue3}. Finally, although it does not play a role in the present paper, we note that due to Remark \ref{rem:continuous}, for some values of $z$ the integral signs must strictly speaking be interpreted as a limit of the meromorphic function in Lemma \ref{lem:residue3}.
\begin{remark} Strictly speaking $\varphi^{\rm \mathsmaller{TR}}$ depends on the path of integration due to possible residues of $\omega_{0,1}$ and $\omega_{\frac{1}{2},1}$, and behaviour of $\omega_{0,2}$ (there is no issue for other terms thanks to Lemma \ref{lem:residue3}). However, this only changes the wavefunction by an overall constant, which has no effect on the quantum curve. Furthermore, since we are integrating with an endpoint at a pole of $ydx$, the first three terms (the coefficients of $\hbar^{-1}$ and $\hbar^0$) in the exponential may diverge. We could remedy this with a regularization method along the same lines as \cite[Remark 4.6]{IKT1}, but since only $z$-derivatives will matter in what follows, we omit the details. 
\end{remark}

\begin{remark}
Note that although refined topological recursion depends on the deformation parameter $\mathscr{Q}$, the wavefunction breaks the $\beta \leftrightarrow 1/\beta$ symmetry. We could have equivalently replaced $\beta \rightarrow 1/\beta$ in the definition, and all arguments in the remainder of the paper would still be valid, but with the appropriate modification.
\end{remark}

For any choice of inverse $z(x)$, we can define a corresponding wavefunction on the base $\mathbb{P}^1$:
\begin{equation}
    \psi^{\rm {TR}}(x):=\varphi^{\rm TR}(z(x)).
\end{equation}
A \emph{quantum curve} is by definition a differential operator\footnote{Geometrically, it is a kind of \emph{oper} \cite{beilinson2005opers,DM2}, acting on $-\frac{1}{2}$-forms, but we will not emphasize this here. See also \cite{iwaki2014exact} for details in the context of WKB analysis.} (for us, of order $2$) such that

\begin{equation}\label{eq:qcurve}
 \left(\hbar^2\dfrac{d^2}{dx^2} + q(x,\hbar) \hbar \dfrac{d}{dx}+r(x,\hbar)\right)\psi^{\rm \mathsmaller{TR}}(x) =0,
\end{equation}
where $q,r$ are formal series in $\hbar$ of meromorphic functions in $x$, with the additional property that its \emph{semiclassical limit} is precisely the equation for the spectral curve:

\begin{equation}
     y^2 + q(x,0)y + r(x,0) = y^2-Q(x).
\end{equation}

Thanks to the form of the wavefunction (that is, the exponential of an $\hbar$-series beginning at order $-1$), it turns out this condition is precisely the condition for $\psi^{\rm \mathsmaller{TR}}$ to be a WKB solution to equation \eqref{eq:qcurve}, which we explain next.

\subsection{WKB analysis}
Let $q(x,\hbar), r(x,\hbar)$ be formal $\hbar$-series of rational functions in $x$. For any Schr\"odinger-like equation

\begin{equation}\label{schrod}
    \left(\hbar^2\dfrac{d^2}{dx^2} + q(x,\hbar) \hbar \dfrac{d}{dx}+r(x,\hbar)\right)\psi =0,
\end{equation}

\noindent we may seek a formal solution, a \emph{WKB solution}, of the form

\begin{equation}\label{wkb-sol}
    \psi^{\mathsmaller{\rm WKB}}(x,\hbar) := \exp{\int^x S(x,\hbar)}= \exp \left( \int^x\dfrac{S_{-1}(x)}{\hbar}+ \int^x S_0(x) + \ldots \right).
\end{equation}

In order for the WKB solution \eqref{wkb-sol} to satisfy (formally) the equation \eqref{schrod}, $S(x,\hbar)$ must then satisfy the \emph{Riccati equation}

\begin{equation}
    \hbar^2 \left(\dfrac{d}{dx}S(x,\hbar)+S(x,\hbar)^2\right)+\hbar \,  q(x,\hbar) S(x,\hbar) + r(x,\hbar) = 0,
\end{equation}
which can be written more explicitly as the recursive equations
\allowdisplaybreaks[0]
\begin{align}
    &(S_{-1})^2+q_0(x)S_{-1}+r_0(x)=0, \label{eq:WKBcurve}\\ 
    &(2S_{-1}+q_0(x))S_{m+1}+\sum_{j=0}^{m}S_{m-j}S_j+q_1(x)S_m+\dfrac{d S_m}{dx} \nonumber \\ & \qquad \qquad \qquad \qquad \;\;\, + \sum_{j=-1}^m q_{m+1-j}(x)S_j + r_{m+2}(x)=0, \qquad m\geq-1.
\end{align}
\allowdisplaybreaks[1]

\noindent Note that the structure of these equations implies that once the solution $S_{-1}$ to the first equation is specified, all the remaining $S_i$ are determined uniquely.

Now fix a spectral curve $(\Sigma,x(z),y(z))$ and a Schr\"odinger-like equation of the form \eqref{schrod}, such that $y=S_{-1}$ satisfies \eqref{eq:WKBcurve}. Then we can pull the WKB solution $\psi^{\rm {WKB}}$ back to the cover $\Sigma$. We denote by $T^{\rm WKB}(z)$ the log-derivative (in $z$) of $\psi^{\rm WKB}(x(z),\hbar)$:

\begin{equation}
T^{\mathsmaller{\rm WKB}}(z):= \dfrac{d}{dz}\log \psi^{\rm {WKB}}(x(z),\hbar)=    S(x(z),\hbar)x'(z).
\end{equation}
Then the Riccati equation becomes,
\allowdisplaybreaks[0]
\begin{align}\label{eq:ricT}
    &(T_{-1}^{\rm \mathsmaller{WKB}})^{2}+U_0(z)T^{{\rm \mathsmaller{WKB}}}_{-1}+V_0(z)=0,\\
    &\left(2T^{\rm \mathsmaller{WKB}}_{-1}+U_0(z)\right)T^{\rm \mathsmaller{WKB}}_{m+1}+\left(\dfrac{d}{dz}-\dfrac{x''(z)}{x'(z)}\right)T^{\rm \mathsmaller{WKB}}_m+ \sum_{j=-1}^m U_{m+1-j}(z)T_{j}^{\rm \mathsmaller{WKB}} \notag \\
    &\hspace{6cm}+\sum_{j=0}^m{T^{\rm \mathsmaller{WKB}}_{m-j}T^{\rm \mathsmaller{WKB}}_j}+V_{m+2}(z)=0, \quad m\geq-1,
\end{align}
\allowdisplaybreaks[1]

\noindent where 
\begin{align}U(z,\hbar):=x'(z)q(x(z),\hbar) \quad V(z,\hbar):=x'(z)^2 r(x(z),\hbar)
\end{align}
are formal $\hbar$-series of rational functions in $z$. Again, once $T^{\rm \mathsmaller{WKB}}_{-1}$ is specified, all other coefficients $T^{\rm \mathsmaller{WKB}}_{m}$, $m\geq 0$, are determined uniquely.

Thus, the statement that a given Schr\"odinger equation \eqref{schrod} is a quantum curve for the spectral curve $(\Sigma,x,y)$ is the statement that the topological recursion wavefunction $\psi^{\rm TR}$ is a WKB solution for a specific choice of $U,V$. 


We will show the following result, which is our main theorem:
\begin{theorem}\label{thm:main}
Let $\mathcal{S}^{\bm \mu}=(\Sigma,x,y,B,D({\bm \mu}))$ be a genus zero degree two refined spectral curve satisfying \emph{Assumption \ref{ass:tr}}. Given a divisor
\begin{equation}
    D(z;{\bm \nu })=[z]-\sum_{p\in \mathcal{P}}\nu_{p}[p], \qquad \sum_{p \in \mathcal{P}}\nu_p=1,
\end{equation}
and any choice of inverse $z(x)$, the wavefunction 
\begin{equation}
    \psi^{\rm \mathsmaller{TR}}(x) = \varphi^{\rm \mathsmaller{TR}}(z(x)),
\end{equation}
is a formal solution of the equation
\begin{equation}
\left(\hbar^{2}\dfrac{d^2}{dx^2}+q(x,\hbar)\hbar \dfrac{d}{dx}+r(x,\hbar)\right)\psi^{\rm \mathsmaller{TR}}(x)=0,
\end{equation}  
where $q(x,\hbar)$, $r(x,\hbar)$ are (finite) formal $\hbar$-series of rational functions in $x$
\begin{equation}
    q(x,\hbar)=q_0(x)+\hbar q_1(x), \quad r(x,\hbar) = r_0(x) + \hbar r_1(x) + \hbar^2 r_2(x).
\end{equation}
which can be explicitly computed. Furthermore, if we denote the unrefined limit $\beta=1$ by $r(x,\hbar;{\bm \nu})=r(x,\hbar;{\bm \mu},{\bm \nu},\beta)|_{\beta=1}$, $q(x,\hbar;{\bm \nu})=q(x,\hbar;{\bm \mu},{\bm \nu},\beta)|_{\beta=1}$, we have
\begin{align}
    &r(x,\hbar;{\bm \mu},{\bm \nu},\beta)=\frac{1}{\beta}r(x,{\hbar \beta^{-\frac12}};\widehat{{\bm \nu}}),\\
    &q(x,\hbar;{\bm \mu},{\bm \nu},\beta)=\frac{1}{\beta^{\frac12}}q(x,{\hbar \beta^{-\frac12}};{{\bm \nu}}),
\end{align}
where $\widehat{\bm \nu}$ denotes the shifted parameters
\begin{equation}
    \hat{\nu}_p={\nu_p}-{\rm ord}_p(y)\left(\tfrac{1+\mu_p}{2}\right)(\beta-1).
\end{equation}
\end{theorem}
\begin{remark} We see that the quantum curve we have found is related in a remarkably simple way to the unrefined quantum curve of \cite{IKT1,IKT2}. In particular, it is obtained by rescaling each term by an appropriate power of $\beta$, together with a shift in the parameters ${\bm \nu}$. Recalling that we defined $\mu_p=-\mu_{\sigma(p)}$ whenever $p \in \widetilde{\mathcal{P}}_-$, since the $p$ for $\nu_p$ appearing in the formula (see (\ref{eq:preciseformula1}--\ref{eq:preciseformula3}) below) for $r_2(x)$  are never ramification points, and the ones appearing in $r_1(x)$ cancel out, the $\hat{\nu}_p$ appearing are indeed defined. 
\end{remark}
\begin{remark}\label{rem:generalcaseqc}
{Even if $x,y$ satisfy a more general algebraic equation of the form $y^2+2b(x) y -Q(x)+b(x)^2=0$, the refined recursion can still be applied by replacing $y$ with $\frac12\Delta y$ as in Remark \ref{rem:ass}, and all $\omega_{g,n}$ are insensitive to $b(x)$ except for $\omega_{0,1}:=ydx$. In this case, the $b(x)$ contributions appear only in the leading order in $\varphi^{\text{TR}}$, and this effect can simply be recovered by the gauge transformation of the wave function $\psi^{\text{TR}}$ by $e^{b(x)/\hbar}$. As a result, the theorem continues to hold verbatim in this case, except that ${\rm ord}_p(y)$ should be replaced with ${\rm ord}_p (\Delta y)$.}
\end{remark}

\subsection{Proof of the recursion relations}
Our approach to determining the quantum curve follows that of \cite{IKT1,IKT2}. In particular, we note that both $\psi^{\rm \mathsmaller{WKB}}$ and $\psi^{\rm \mathsmaller{TR}}$ are (exponentials of) formal $\hbar$-series of the same form, and thus it is sufficient to find $q(z)=q(x(z))$ and $r(z)=r(x(z))$ which factor through $x$, such that 
\begin{equation}
T^{\rm \mathsmaller{TR}}(z):=\sum_{m\geq-1}\hbar^m T^{\rm \mathsmaller{TR}}_m(z):=\frac{d}{dz}\log\psi^{\rm \mathsmaller{TR}}
\end{equation} satisfies the Riccati equation \eqref{eq:ricT} in the $z$-variable. The rest of this section is devoted to determining such functions.

\subsubsection*{Notation}
We will use a similar notation to \cite{IKT1,IKT2}. We let $\mathpzc{G}_{g,n}$ denote the $1$-differential defined as the $n$ times integral of $\omega_{g,n+1}$:
\begin{equation}
    \mathpzc{G}_{g,n+1}(z_0,z_1,\ldots,z_n):= \begin{cases}\omega_{g,1}(z_0),\hfill n=0\vspace{8pt} \\
    
    {\displaystyle\int}_{\!\!\!D(z_1;{\bm \nu})}\!\!\!\ldots {\displaystyle\int}_{\!\!\!D(z_{n};{\bm \nu})} \omega_{g,n+1}(z_0,\zeta_1,\ldots \zeta_{n}),\qquad n>0
    \end{cases}
\end{equation}
where the $i$th integration is taken in the $\zeta_i$ variable. We also define $\mathpzc{H}_{g,n}$ by specializing all variables to $z$ in $\mathpzc{G}_{g,n}$:

\begin{equation}
    \mathpzc{H}_{g,n}(z):=\dfrac{1}{(n-1)!}\mathpzc{G}_{g,n}(z_1,\ldots,z_{n})\bigg|_{z_1=\,\ldots\, =z_{n}=z}, \quad (g,n)\neq(0,2)
\end{equation}

\noindent Throughout, we take $\mathpzc{H}_{g,n}=0$ by convention if $g<0$ or $n<1$. We can then write

\begin{equation}\label{eq:TmTRH}
    T_m^{\rm \mathsmaller{TR}}(z)dz = \sum_{\substack{2g-2+n=m \\ g\geq0,n\geq1}} \beta^{-\frac{n}{2}}\mathpzc{H}_{g,n}(z), \qquad m\neq0
\end{equation}
i.e. $T^{\rm \mathsmaller{TR}}_m dz$ is the sum of all $\mathpzc{H}_{g,n}$ at level $m=2g-2+n$.  For the case $m=0$, we have
\begin{equation}
     T_0^{\rm \mathsmaller{TR}}dz= -\dfrac{1}{\beta }\mathpzc{G}_{0,2}(\sigma({z}),z)+\dfrac{1}{\sqrt{\beta}}\mathpzc{G}_{\frac{1}{2},1}(z).
\end{equation}




We now proceed to obtain recursion relations for $\mathpzc{G},\mathpzc{H}$, and $T^{\rm \mathsmaller{TR}}$. For $\mathpzc{G}_{g,n+1}$, we have:

\begin{lemma}[Recursion for $\mathpzc{G}_{g,n}$]
For any $(g,n)$ with $2g-2+n > 1$, we have
\begin{align}\label{eq:grec}
\notag \mathpzc{G}_{g,n+1}(z_0,J)=&-\frac{1}{{2 y}(z_0)dx(z_0)} \int_{D(z_1;{\bm \nu})}\!\!\!\ldots \int_{D(z_n;{\bm \nu})}\omega_{g-1,n+2}(z_0,z_0,\zeta_J) \\ \notag
&- \sum_{j=1}^n \left(\frac{\mathpzc{G}_{0,2}(z_0,z_j)\!-\!\mathpzc{G}_{0,2}(\sigma(z_0),z_j)}{{2 y}(z_0)dx(z_0)}\mathpzc{G}_{g,n}(z_0,\widehat{J}_{j}) - \dfrac{\mathpzc{G}_{0,2}(z_0,z_j)\!-\!\mathpzc{G}_{0,2}(z_0,\sigma(z_j))}{{2 y}(z_j)dx(z_j)}\mathpzc{G}_{g,n}(z_j,\widehat{J}_{j}) \right)\\
\notag &-\frac{1}{{2 y}(z_0)dx(z_0)}\sum^{**}_{\substack{g_1+g_2=g \\ J_1 \sqcup J_2=J}} \mathpzc{G}_{g_1,n_1+1}(z_0,J_1)\cdot \mathpzc{G}_{g_2,n_2+1}(z_0,J_2)\\
& - \frac{\mathscr{Q}}{{2 y}(z_0)dx(z_0)}dx(z_0) \cdot d_{z_0}\left(\dfrac{\mathpzc{G}_{g-\frac{1}{2},n+1}(z_0,J)}{dx(z_0)}\right).
\end{align}
\end{lemma}

\begin{proof}

In the proof of Lemma \ref{lem:poles}, after taking residues in the recursion formula \eqref{eq:RTR} we obtained the expression \eqref{pole2}
\begin{equation}
    \omega_{g,n+1}(z_0,J)=-\frac{\text{Rec}^{\mathscr{Q}}_{g,n+1}(z_0,J)}{2\omega_{0,1}(z_0)} +2\sum_{i=1}^n d_{z_i}\left({\rm K}(z_0,z_i)\cdot\omega_{g,n}(J)\right).
\end{equation}
Integrating both sides and rearranging, we obtain the result.

\end{proof}


We note that the last term can also be written, using the notation $\partial_z(f(z)dz):=\partial_z f(z)dz$,
\begin{align}\label{eq:gproductrule}
    &-\frac{\mathscr{Q}}{{2 y}(z_0)}{d_{z_0}}\!\!\left(\frac{\mathpzc{G}_{g-\frac{1}{2},n+1}(z_0,z_J)}{dx(z_0)}\right)\notag\\
    &\hspace{2cm}=-\dfrac{\mathscr{Q}}{{2 y}(z_0)x'(z_0)}\dfrac{\partial}{\partial z_0}\mathpzc{G}_{g-\frac{1}{2},n+1}(z_0,z_J)+\frac{\mathscr{Q} x''(z_0)}{{2 y}(z_0)x'(z_0)^2}\mathpzc{G}_{g-\frac{1}{2},n+1}(z_0,z_J)
    \end{align}

Specializing all variables to $z$, it is easy to deduce the recursion for $\mathpzc{H}_{g,n+1}$ from the recursion for $\mathpzc{G}_{g,n+1}$:
\begin{lemma}[Recursion for $\mathpzc{H}_{g,n}$]
\label{lem:hrec}
For any $(g,n)$ with $2g-2+n > 1$, we have
\allowdisplaybreaks[0]
\begin{align}\label{eq:hreceq}
    \mathpzc{H}_{g,n+1}(z)=&-\frac{1}{{2 y}(z)x'(z)}\frac{\partial}{\partial z_0}\biggl(\mathpzc{H}_{g-1,n+2}(z_0)-\frac{\mathpzc{G}_{g-1,n+2}(z_0,z,\ldots,z)}{(n+1)!}+\frac{\mathpzc{G}_{g,n}(z_0,z,\ldots,z)}{(n-1)!}\notag\\
    &\hspace{8.4cm}+\mathscr{Q}\,\frac{{\mathpzc{G}_{g-\frac{1}{2},n+1}(z_0,z,\ldots z)}}{n!}\biggr)\bigg|_{z_0=z} \notag \\
    &-\frac{d}{dz}\left(\frac{1}{{2 y}(z)x'(z)}    \right) \mathpzc{H}_{g,n}(z)+ \frac{\mathpzc{G}_{0,2}(\sigma(z),z)-\mathpzc{G}_{0,2}(z,\sigma(z))}{{2 y}(z)dx(z)}\mathpzc{H}_{g,n}(z)  \notag\\
    &+\frac{\mathscr{Q} x''(z)}{{2 y}(z)x'(z)^2}\mathpzc{H}_{g-\frac{1}{2},n+1}(z)-\frac{1}{{2 y}(z)dx(z)}\sum^{**}_{\substack{g_1+g_2=g\\n_1+n_2=n}} \mathpzc{H}_{g_1,n_1+1}(z)\cdot \mathpzc{H}_{g_2,n_2+1}(z).
\end{align}
\end{lemma}
\allowdisplaybreaks[1]

\begin{proposition}[Recursion for $T_{m}^{\rm \mathsmaller{TR}}$]\label{prop:trec}
For any $m\geq1$, we have
\begin{align}\label{eq:trec}
    &2T_{-1}^{\rm \mathsmaller{TR}}T_{m+1}^{\rm \mathsmaller{TR}}+\left(\frac{d}{dz}-\frac{1}{\beta}\left(\frac{ y'(z)}{y(z)}-\frac{{\mathpzc{G}_{0,2}(\sigma(z),z)+\mathpzc{G}_{0,2}(z,\sigma(z))}}{dz}\right)-\frac{x''(z)}{x'(z)}\right)T_m^{\rm \mathsmaller{TR}} \notag\\
    &\hspace{9cm}+ \sum_{j=0}^m\mathit{T}^{\rm \mathsmaller{TR}}_{m-j}\mathit{T}^{\rm \mathsmaller{TR}}_{j}=0.
\end{align}
\end{proposition}
\begin{proof}
From \eqref{eq:TmTRH}, we have

\begin{equation}
    T^{{\rm \mathsmaller{TR}}}_{m+1}(z)dz=\sum_{\substack{2g-2+n=m\\g\geq0,n\geq0}}{\beta^{-\frac{n+1}{2}}} \, \mathpzc{H}_{g,n+1}(z), \qquad m\geq1.
\end{equation}
Rewriting the right-hand side with the $\mathpzc{H}$-recursion, we consider the sum appearing inside the derivative of \ref{eq:hreceq}. Recalling $\mathscr{Q}=\beta^{\frac{1}{2}}-\beta^{-\frac{1}{2}}$, we can write this part of $T_{m+1}^{\rm \mathsmaller{TR}}(z)dz$ as

\begin{align}
    &\sum_{\substack{2g-2+n=m\\g\geq0,n\geq0}} \frac{1}{\beta^{\frac{n+1}{2}}}\Bigg(\mathpzc{H}_{g-1,n+2}(z_0)-\dfrac{\mathpzc{G}_{g-1,n+2}(z_0,z,\ldots,z)}{(n+1)!}+\dfrac{\mathpzc{G}_{g,n}(z_0,z,\ldots,z)}{(n-1)!}\nonumber\\ 
     &\hspace{6cm}+\left({\beta^{\frac{1}{2}} -\beta^{-\frac{1}{2}}}\right)\frac{\mathpzc{G}_{g-\frac{1}{2},n+1}(z_0,z,\ldots,z)}{n!}\Bigg) \label{eq:hggg1}\\
     &=\sum_{\substack{2g-2+n=m\\g\geq0,n\geq2}}\frac{1}{\beta^{\frac{n-1}{2}}} \mathpzc{H}_{g,n}(z_0)+\mathpzc{G}_{\frac{m+1}{2},1}(z_0)\label{eq:hggg2}\\
    &=\beta^{\frac{1}{2}}\sum_{\substack{2g-2+n=m\\g\geq0,n\geq1}}\beta^{-\frac{n}{2}}\mathpzc{H}_{g,n}(z_0)\\
    &=\beta^{\frac{1}{2}}\,T_m^{\rm \mathsmaller{TR}}(z_0)dz_0,
\end{align}
where we passed from \eqref{eq:hggg1} to \eqref{eq:hggg2} with a careful manipulation of summation indices.

Next we treat the product term in the recursion, and write it in terms of $T^{\rm \mathsmaller{TR}}_m$: 

\begin{align}
    & \sum^{**}_{\substack{g_1+g_2=g\\n_1+n_2=n\\2g-2+n=m}}\beta^{-\frac{n+1}{2}}\mathit{\mathpzc{H}}_{g_1,n_1+1}(z)\mathit{\mathpzc{H}}_{g_2,n_2+1}(z)= 
     \beta^{\frac12}\sum_{j=0}^m\mathit{T}^{\rm \mathsmaller{TR}}_{m-j}(z)\mathit{T}^{\rm \mathsmaller{TR}}_{j}(z) + \frac{1}{\beta^{\frac12}}2\mathpzc{G}_{0,2}(\sigma({z}),z)\mathit{T}^{\rm \mathsmaller{TR}}_m(z).
\end{align}

We combine these with the $\mathpzc{H}$-recursion of Lemma \ref{lem:hrec} to obtain the recursion for $T_m^{\rm \mathsmaller{TR}}$:
\begin{align}
    T^{\rm \mathsmaller{TR}}_{m+1}(z)=&-\beta^{\frac12}\frac{1}{2 y(z) x'(z)}\frac{\partial}{\partial z}T^{\rm \mathsmaller{TR}}_{m}(z)-\frac{1}{\beta^{\frac12}}\left(\frac{\partial}{\partial z}\frac{1}{2 y(z)x'(z)}\right)T^{\rm \mathsmaller{TR}}_m(z)\nonumber\\ &-\beta^{\frac12}\frac{1}{2 y(z) x'(z)}\sum_{j=0}^m\mathit{T}^{\rm \mathsmaller{TR}}_{m-j}(z)\mathit{T}^{\rm \mathsmaller{TR}}_{j}(z)  +\frac{1}{\beta^{\frac12}}\frac{\mathpzc{G}_{0,2}(\sigma(z),z)-\mathpzc{G}_{0,2}(z,\sigma(z))}{{2 y}(z)dx(z)}T_m^{\rm \mathsmaller{TR}}\nonumber\\
    &-\frac{1}{\beta^{\frac12}}\frac{1}{ y(z) dx(z)}\mathpzc{G}_{0,2}(\sigma({z}),z)\mathit{T}^{\rm \mathsmaller{TR}}_m(z) +\frac{\mathscr{Q}}{2 y(z)x'(z)}\frac{x''(z)}{x'(z)}T_{m}^{\rm \mathsmaller{TR}}(z).
    \end{align}
    Recalling $\mathscr{Q}=\beta^{\frac12}-\beta^{-\frac12}$, the last terms of the first and third lines combine nicely and we can write this as
    \begin{align}
    =&-\beta^{\frac12}\frac{1}{2 y(z) x'(z)}\frac{d}{d z}T^{\rm \mathsmaller{TR}}_{m}(z)+\frac{1}{\beta^{\frac12}}\left(\frac{y ' (z)}{2 y(z)^2 x'(z)}\right)T^{\rm \mathsmaller{TR}}_m(z)\nonumber\\ &-\beta^{\frac12}\frac{1}{2 y(z) x'(z)}\sum_{j=0}^m\mathit{T}^{\rm \mathsmaller{TR}}_{m-j}(z)\mathit{T}^{\rm \mathsmaller{TR}}_{j}(z)  -\frac{1}{\beta^{\frac12}}\frac{\mathpzc{G}_{0,2}(\sigma(z),z)+\mathpzc{G}_{0,2}(z,\sigma(z))}{{2 y}(z)dx(z)}T_m^{\rm \mathsmaller{TR}}\nonumber\\
    &\hspace{7cm}+\beta^{\frac12}\frac{1}{2 y(z)x'(z)}\frac{x''(z)}{x'(z)}T_{m}^{\rm \mathsmaller{TR}}(z).
\end{align}
Multiplying by $T_{-1}^{\rm \mathsmaller{TR}}=\frac{1}{\beta^{\frac12}}y(z)x'(z)$ and rearranging, we obtain the result.\end{proof}

\begin{proof}[Proof of Theorem \ref{thm:main}]
Due to Assumption \ref{ass:rtr}, the semiclassical limit condition in the definition of quantum curve implies that $q_0$ (and $U_0$) vanishes, and that $V_0(x)=-\beta^{-1}x'(z)^2Q(x(z))$. Then comparing the ($m\geq1$) recursion \eqref{eq:trec} for $T^{\rm \mathsmaller{TR}}$ with the recursion \eqref{eq:ricT} for $T^{\rm \mathsmaller{WKB}}$ we immediately see that demanding the recursion relations are identical requires that we set all $U_j(z)=0$ for $j>1$ and $V_j(z)=0$ for $j>2$; matching the terms linear in $T_m$ gives $U_1(z)$ (and thus $q_1$).

The remaining equations \eqref{eq:ricT} for $m=-1,0$ give formulas for $V_1$ and $V_2$ (hence $r_1$, $r_2$) in terms of $T^{{\rm \mathsmaller{TR}}}_{-1}$, $T^{{\rm \mathsmaller{TR}}}_{0}$, $T^{{\rm \mathsmaller{TR}}}_{1}$ (whose recursions we omitted) and $U_1$, which we already determined. In terms of $z$, explicitly, we arrive at the following expressions after tedious computations:
\begin{align} 
    &q_0=0,\qquad\qquad\quad\;\;\;  q_1=\frac{1}{\beta x'(z)}\left(-\frac{ y'(z)}{ y(z)}+\frac{2}{z-\sigma(z)} - \sum_{p \in \mathcal{P}}\frac{{\nu}_p+{\nu}_{\sigma(p)}}{z-p}\right)\label{eq:preciseformula1}\\
    &r_0=-\frac{1}{\beta}Q(x(z)),\qquad r_1=\frac{ y(z)}{\beta^{3/2}x'(z)}\sum_{p\in\mathcal{P}}\frac{\hat{\nu}_p-\hat{\nu}_{\sigma(p)}}{z-p},\label{eq:preciseformula2}\\
    &r_2= \frac{y(z)}{\beta^{2}x'(z)}\sum_{p\in \mathcal{P}^{(1)}}\frac{\hat{\nu}_p\,\hat{\nu}_{\sigma(p)}}{\underset{w=p}{{\rm Res}}\,{ y(w)dx(w)}}\frac{1}{(z-p)},\label{eq:preciseformula3}
\end{align}
where 
\begin{equation}
    \hat{\nu}_p={\nu_p}-{\rm ord}_p(y)\left(\tfrac{1+\mu_p}{2}\right)(\beta-1).
\end{equation}

It remains to show that $q_1$, and $r_1$, $r_2$ are indeed well-defined as rational functions of $x$ --- in particular, we must check these expressions are invariant under the involution. Since the formulas (\ref{eq:preciseformula1}--\ref{eq:preciseformula3}) are of exactly the same form as the unrefined case up to the shift and $\beta$ factors, this follows from the result of \cite{IKT1,IKT2}.
\end{proof}


\section{Examples and the Nekrasov-Shatashvili limit}

In this section, we apply our results to an interesting collection of explicit examples of genus zero degree two curves, which played a central role in \cite{IKT1,IKT2,Iwaki:2021zif,IWAKI2022108191}. We write their quantum curves explicitly, and take the so-called Nekrasov-Shatashvili limit, observing a simple interpretation of the parameters ${\bm \mu}$ in this case.

\subsection{Spectral curves of hypergeometric type}

Although we have formulated and studied the recursion for genus zero degree two refined spectral curves, in fact we have been motivated by a set of examples arising from the ``Gauss hypergeometric'' curve $\Sigma$ defined by:

\begin{equation}
y^2= \dfrac{m_\infty^2 x^2-(m_\infty^2+m_0^2-m_1^2)x+m_0^2}{x^2(x-1)^2}
\end{equation}
and various limits of it. In particular, the corresponding quantum curve in this case turns out to the classical Gauss hypergeometric equation, and for the various other examples we obtain various confluent limits (more precise descriptions of their relationships can be found in e.g. \cite{IKT2}). Hence, we refer to such examples as \emph{(refined) spectral curves of hypergeometric type}.

Explicitly and by definition, these are nine (families of) curves of the form
\begin{equation}
    y^2=Q(x)
\end{equation}
given by the rational functions listed in Table \ref{table:classical} below, together with the functions $x$ and $y$ (and the canonical choice \eqref{eq:canonicalB} for $B$). It is easy to verify that these curves, together with a choice of divisor $D({\bm \mu})$ are all examples of genus zero degree two refined spectral curves, so that all of our results thus far may be applied.

\begin{table}[h]
\begin{center}
\begin{tabular}{ccc}\hline
${\rm Equation}$ & $Q(x)$  & Assumption
\\\hline
\parbox[c][4.0em][c]{0em}{}
${\rm Gauss}$
&
\begin{minipage}{.35\textwidth}
\begin{center}
$\dfrac{{m_\infty}^2 x^2 
- ({m_\infty}^2 + {m_0}^2 - {m_1}^2)x 
+ {m_0}^2}{x^2 (x-1)^2}$
\end{center}
\end{minipage}
& ~~~\quad
\begin{minipage}{.35\textwidth}
\begin{center}
$m_0, m_1, m_\infty \neq 0$,\\
$m_0 \pm m_1 \pm m_\infty \ne 0$.
\end{center}
\end{minipage}
\\\hline
\parbox[c][3.0em][c]{0em}{}
${\rm Degenerate~Gauss}$
&
\begin{minipage}{.35\textwidth}
\begin{center}
$\dfrac{{m_\infty}^2 x + {m_1}^2 - {m_\infty}^2}{x(x-1)^2}$
\end{center}
\end{minipage}
&
\begin{minipage}{.35\textwidth}
\begin{center}
$m_1, m_\infty \neq 0$,\\
$m_1 \pm m_\infty \neq 0$.
\end{center}
\end{minipage}
\\\hline
\parbox[c][3.0em][c]{0em}{}
${\rm Kummer}$ 
&
\begin{minipage}{.3\textwidth}
\begin{center}
$\dfrac{x^2 + 4 m_\infty x + 4 {m_0}^2}{4x^2}$
\end{center}
\end{minipage}
&
\begin{minipage}{.3\textwidth}
\begin{center}
$m_0\neq 0$,\\
$m_0 \pm m_\infty \neq 0$.
\end{center}
\end{minipage}
\\\hline
\parbox[c][3.0em][c]{0em}{}
${\rm Legendre}$ 
& $\dfrac{m_\infty^2}{x(x-1)}$
&
$m_\infty \neq 0$.
\\\hline
\parbox[c][3.0em][c]{0em}{}
${\rm Bessel}$ 
& $\dfrac{x + 4m_0^2}{4x^2}$
&
$m_0 \neq 0$.
\\\hline
\parbox[c][3.0em][c]{0em}{}
${\rm Whittaker}$ 
& $\dfrac{x + 4m_\infty}{4 x}$
& $m_\infty \neq 0$.
\\\hline
\parbox[c][3.0em][c]{0em}{}
${\rm Weber}$ 
& $\dfrac{1}{4} x^2 - m_\infty$
& $m_\infty \neq 0$.
\\\hline
\parbox[c][3em][c]{0em}{}
${\rm Degenerate~Bessel}$ 
& $\dfrac{1}{4x}$
& --
\\\hline
\parbox[c][2.5em][c]{0em}{}
${\rm Airy}$ 
& $x$
& --
\\\hline
\end{tabular}
\end{center}
\caption{Spectral curves $\Sigma : y^2 - Q(x) = 0$ 
of hypergeometric type.} 
\label{table:classical}
\vspace{-7mm}
\end{table}
Each curve comes in a family parametrized by a collection of complex numbers $\{m_s\}$ associated to each distinct $s=x(p)$ for $p\in\mathcal{P}\setminus\mathcal{R}$. We will use the notation $s_{\pm}$ for the two points in $x^{-1}({s})$, with convention that
\begin{equation}
    \underset{q=s_{\pm}}{\rm Res}\,y(q)dx(q) = \pm m_{s}.
\end{equation}
We call the parameters $\{m_s\}$ the \emph{masses}.

\subsection{Nekrasov-Shatashvili limit}

Since the $\omega_{g,n}$ are polynomials in $\mathscr{Q}$ of degree ${2g}$, counting the powers of $\hbar$ and $\beta$ in the definition of the wavefunction we see that we can take the \emph{Nekrasov-Shatashvili (NS) limit} (\cite{NS1, Nekrasov:2011bc}) as $\hbar\rightarrow 0$, $\beta\rightarrow \infty$ with $\hbar \beta^{\frac12}=\epsilon_1$ fixed:
\begin{equation}
    \varphi^{\rm NS}(z):= \underset{\substack{\hbar\rightarrow 0\\ \beta\rightarrow\infty \\ \hbar\beta^{\frac12} = \epsilon_1}}{\lim} \varphi^{\rm TR}(z),
\end{equation}
to define the \emph{Nekrasov-Shatashvili wavefunction} $\psi^{\rm NS}(x)$. We may again ask for quantum curves which annihilate this wavefunction, and it is easy to check that this commutes with the quantization, so we may simply take the limit to obtain the NS quantum curve (see \cite{BMT} for discussion from the matrix model perspective).

In the physics literature, this limit plays an important role and enjoys a close relation to quantum integrability and other applications. In fact, in light of the relation to topological string theory, one may identify $\epsilon_1=\hbar\beta^{\frac12}$, $\epsilon_2=-\hbar \beta^{-\frac12}$, corresponding to the so-called \emph{$\Omega$-background} parameters of the associated gauge theory (so that the NS limit becomes $\epsilon_2\rightarrow0$, and the self-dual/unrefined limit $\epsilon_1+\epsilon_2\rightarrow0$).  For comparison to such a setting, we give the explicit form of the corresponding quantum curves below in terms of these parameters.

In the Nekrasov-Shatashvili limit, the examples we present below clarify the effect of the new parameters ${\bm \mu}$: in particular, the quantum curves are all of the form
\begin{equation}
\left(\epsilon_1^2\frac{d^2}{dx^2}+\widetilde{Q}(x,\epsilon_1)\right)\psi^{\rm NS}(x)=0,\label{eq:NSform}
\end{equation}
where $\widetilde{Q}(x,\epsilon_1)$ is obtained from $Q(x)$ by replacing masses with a simple $\epsilon_1$ and ${\bm \mu}$-dependent shift, explicitly written below. The parameters ${\bm \nu}$ no longer appear in this limit. In fact, after gauge transformation, the parameters $\nu_p$ for $p\in\widetilde{\mathcal{P}}$ in the self-dual limit \cite{IKT1,IKT2} and the parameters $\mu_p$ in the NS limit below play a similar role.

\subsection{Explicit expressions of quantum curves}
\label{sec:QCHG}

Finally, we present the concrete form of the quantum curves of all the examples of hypergeometric type. We first write down the the explicit formulas and basic properties for the spectral curves in each example in Table \ref{table:classical}, which can then be used to compute the output of the refined topological recursion. When $\beta=1$, we naturally reproduce all quantum curves obtained by Iwaki-Koike-Takei \cite{IKT2} (up to a minor change in convention for the Legendre, Whittaker, and degenerate Bessel curves).

We let $z$ denote an affine coordinate for a chart of $\mathbb{P}^1$.


\subsubsection*{Gauss hypergeometric:}



The functions defining the spectral curve are given by
\begin{align}
    &x(z)= \frac{\sqrt{M } \left(z-{{\alpha_{0,+}}}\right) (z-{\alpha_{0,-}})}{4 z m _{\infty }^2}\\
    &y(z)=\frac{4 z \left(z^2-1\right) m _{\infty }^3}{\sqrt{M } \left(z-\alpha_{0,+}\right) (z-{\alpha_{0,-}}) \left(z-{{\alpha_{1,+}}}\right) (z-{\alpha_{1,-}})}
\end{align}
where 
\begin{equation}
    \alpha_{0_\pm}=-\frac{(m_0\pm m_\infty)^2-m_1^2}{\sqrt{M}}, \qquad \alpha_{1_\pm}=\frac{(m_1\pm m_\infty)^2-m_0^2}{\sqrt{M}}
\end{equation}
and
\begin{equation}
    M = (m_0+m_1+m_\infty)(m_0+m_1-m_\infty)(m_0-m_1+m_\infty)(m_0-m_1-m_\infty).
\end{equation}
We have the involution $\sigma(z)=1/z$. The (effective) ramification points are $\mathcal{R}=\mathcal{R}^*=\{+1,-1\}$, and $\mathcal{P}=\widetilde{\mathcal{P}}=\{0,\infty,\alpha_{0_+},\alpha_{0_-},\alpha_{1_+},\alpha_{1_-}\}$, with $\widetilde{\mathcal{P}}_+=\{0,\alpha_{0_+},\alpha_{1_+}\}$.


 The quantum curve is explicitly given by:
\begin{align}\label{eq:qcgauss}
   &  \Bigg(\epsilon_1^2\dfrac{d^2}{dx^2}-\epsilon_1\epsilon_2 \bigg(\dfrac{1-\nu_{0_+}-\nu_{0_-}}{ x}+\dfrac{1-\nu_{1_+}-\nu_{1_-}}{(x-1)}\bigg) \dfrac{d}{dx} \nonumber\\
    &-\frac{ \left(\epsilon_2\nu_{0_+}-(\epsilon_1\!\!+\!\epsilon_2)\frac{1+\mu _{{0_+}}}{2}\right)\left(\epsilon_2\nu _{0_-}\!-(\epsilon_1\!\!+\!\epsilon_2)\frac{1-\mu _{{0_+}}}{2}\right)}{x^2(x-1)} +\frac{m_0 \left(\epsilon_2(\nu _{0_+}\!\!-\nu_{0_-})-(\epsilon_1\!\!+\!\epsilon_2)\mu _{{0_+}} \right)}{x^2(x-1) }\nonumber\\
    &+\frac{ \left(\epsilon_2\nu_{1_+}-(\epsilon_1\!\!+\!\epsilon_2)\frac{1+\mu _{{1_+}}}{2}\right)\left(\epsilon_2\nu _{1_-}-(\epsilon_1\!\!+\!\epsilon_2)\frac{1-\mu _{{1_+}}}{2}\right)}{x(x-1)^2} -\frac{m_1 \left(\epsilon_2(\nu _{1_+}\!\!-\nu_{1_-})-(\epsilon_1\!\!+\!\epsilon_2)\mu _{{1_+}}\right)}{x(x-1)^2 }\nonumber\\
    &+\frac{\left(\epsilon_2\nu_{\infty_+}+(\epsilon_1\!\!+\!\epsilon_2)\frac{1+\mu _{{\infty_+}}}{2}\right)\left(\epsilon_2\nu_{\infty_-}+(\epsilon_1\!\!+\!\epsilon_2)\frac{1-\mu _{{\infty_+}}}{2}\right)}{x(x-1)} -\frac{m_\infty \left( \epsilon_2(\nu _{\infty_+}\!\!-\nu_{\infty_-})+(\epsilon_1\!\!+\!\epsilon_2)\mu _{{\infty_+}}\right)}{x(x-1) }\nonumber\\
    &\hspace{7.75cm}-\dfrac{m_\infty^2 x^2-(m_\infty^2+m_0^2-m_1^2)x+m_0^2}{x^2(x-1)^2}\Bigg)\psi(x)=0
\end{align}
In the NS limit, this can be written as

\begin{align}\label{eq:qcgaussNS}
    &\Bigg(\epsilon_1^2\dfrac{d^2}{dx^2}
-\dfrac{{\rm m}_\infty^+ {\rm m}_{\infty}^- x^2-({\rm m}_\infty^{+} {\rm m}_\infty^{-}+{\rm m}_0^{+} {\rm m}_0^{-}-{\rm m}_1^{+} {\rm m}_1^{-})x+{\rm m}_0^+ {\rm m}_0^-}{x^2(x-1)^2}\Bigg)\psi^{\rm NS}(x)=0
\end{align}
where ${\rm m}_{0}^{\pm}=m_0-\frac{\epsilon_1\mu_{0_+} }{2}\pm \frac{\epsilon_1}{2}$, ${\rm m}_{1}^{\pm}=m_1-\frac{\epsilon_1\mu_{1_+} }{2}\pm \frac{\epsilon_1}{2}$, and ${\rm m}_{\infty}^{\pm}=m_\infty+\frac{\epsilon_1\mu_{\infty_+} }{2}\mp \frac{\epsilon_1}{2}$. 

\subsubsection*{Degenerate hypergeometric:}

The functions defining the spectral curve are given by
\begin{align}
    x(z)= \frac{m_1^2-m_{\infty }^2}{z^2-m_{\infty }^2},\qquad y(z)=-\frac{z \left(z^2-m_{\infty }^2\right)}{z^2-m_1^2}
\end{align}
with the involution $\sigma(z)=-z$. The (effective) ramification points are $\mathcal{R}=\mathcal{R}^*=\{0,\infty\}$, and we have $\mathcal{P}=\widetilde{\mathcal{P}}=\{m_1,-m_1,m_\infty,-m_\infty\}$, with $\widetilde{\mathcal{P}}_+=\{m_1,-m_\infty\}$.\medskip

\noindent The quantum curve is identical to the Gauss hypergeometric case \eqref{eq:qcgauss} with the parameters $m_0$ and $\nu_{0_\pm}$ set to zero, and $\mu_{0_+}=\pm 1$.

\subsubsection*{Kummer:}

The functions defining the spectral curve are given by
\begin{align}
    x(z)= \sqrt{m_{\infty }^2-m _0^2}\left(z+\frac{1}{z}\right) -2 m_{\infty },\qquad y(z)=\frac{z^2-1}{2 \left(z-{\alpha_+} \right) \left(z-\alpha_-\right)}
\end{align}
where 
\begin{equation}
\alpha_{0_\pm}=\frac{m_\infty \pm m_0}{\sqrt{m_\infty^2-m_0^2}}.
\end{equation}
We have the involution $\sigma(z)=1/z$. The (effective) ramification points are $\mathcal{R}=\mathcal{R}^*=\{+1,-1\}$, and we have ${\mathcal{P}}=\{0,\infty,\alpha_{0_+},\alpha_{0_-}\}$, $\widetilde{\mathcal{P}}=\{\alpha_{0_+},\alpha_{0_-}\}$, with $\widetilde{\mathcal{P}}_+=\{\alpha_{0_+}\}$.

\noindent The quantum curve is explicitly given by:
\begin{align}\label{eq:qckummer}
&\Bigg(  \epsilon_1 ^2 \dfrac{d^2}{dx^2} - \epsilon_1\epsilon_2 \frac{\nu _{\infty _+}\!\!+\nu _{\infty _-}}{  x}\frac{d}{dx}  + \frac{\left(\epsilon_2\nu_{0_+}\!\!-(\epsilon_1\!\!+\!\epsilon_2)\frac{1+\mu _{{0_+}}}{2}\right)\left(\epsilon_2\nu _{0_-}\!\!-(\epsilon_1\!\!+\!\epsilon_2)\frac{1-\mu _{{0_+}}}{2}\right)}{ \mathit{x}^2}\nonumber \\
&  -\left(\frac{\epsilon_2(\nu _{\infty _+}\!\!-\nu _{\infty _-})}{2  \mathit{x}} + \frac{m_0 \left(\epsilon_2(\nu _{0_+}\!\!-\nu_{0_-})-(\epsilon_1\!\!+\!\epsilon_2)\mu _{{0_+}} \right)}{ \mathit{x}^2}\right)-\frac{x^2+4 m_{\infty }x+4 m_0^2}{4  x^2}\Bigg)\psi^{\rm TR}(x)=0
\end{align}
We note that in particular, even though $ydx$ has poles at $0$ and $\infty$, the corresponding parameters $\nu_{\infty_{\pm}}$ are unshifted from the unrefined case since $y$ is regular there.

In the Nekrasov-Shatashvili limit, the quantum curve can be written
\begin{align}
&\Bigg(  \epsilon_1^2 \dfrac{d^2}{dx^2}  -\frac{x^2+4  m_{\infty }x+4 {\rm m}_0^+{\rm m}_0^-}{4 x^2} \Bigg)\psi^{\rm NS}(x)=0.
\end{align}
where ${\rm m}_{0}^{\pm}=m_0-\frac{\epsilon_1\mu_{0_+} }{2}\pm \frac{\epsilon_1}{2}$.

\subsubsection*{Legendre:}





The functions defining the spectral curve are given by
\begin{align}
    x(z)= \frac{1}{2} \left(1+z+\frac{1}{z}\right),\qquad y(z)=\frac{4 \, m_{\infty } z}{z^2-1}
\end{align}
with the involution $\sigma(z)=1/z$. The (effective) ramification points are $\mathcal{R}=\mathcal{R}^*=\{+1,-1\}$, and we have $\widetilde{\mathcal{P}}=\{0,\infty\}$, with $\widetilde{\mathcal{P}}_+=\{0\}$.\medskip

\noindent The quantum curve is identical to the Gauss hypergeometric case\eqref{eq:qcgauss} with parameters $m_0$, $m_1$, $\nu_{0_\pm}$, and $\nu_{1_\pm}$ set to zero, and $\mu_{0_+}=\pm 1$, $\mu_{1_+}=\pm1$.

\subsubsection*{Bessel:}

The functions defining the spectral curve are given by
\begin{align}
    x(z)=4 m_0^2 \left(z^2-1\right),\qquad y(z)= \frac{z}{4 m_0 \left(z^2-1\right)}
\end{align}
with the involution $\sigma(z)=-z$. The (effective) ramification points are $\mathcal{R}=\{0,\infty\}$, $\mathcal{R}^*=\{0\}$ and we have $\mathcal{P}=\widetilde{\mathcal{P}}=\{-1,1\}$ with $\widetilde{\mathcal{P}}_{+}= \{1\}$.\medskip

\noindent The quantum curve is explicitly given by:

\begin{align}\label{eq:qcbessel}
  &\Bigg(\epsilon_1^2 \dfrac{d^2}{dx^2} +\epsilon_1\epsilon_2 \frac{ (\nu _{0_+}\!\!+\nu _{0_-}\!\!-1)}{x}\dfrac{d}{dx}  + \frac{\bigl(\epsilon_2\nu_{0_+}\!\!-(\epsilon_1\!\!+\!\epsilon_2)\frac{1+\mu_{0_+}}{2}\bigr)\bigl(\epsilon_2\nu _{0_-}\!\! -(\epsilon_1\!\!+\!\epsilon_2)\frac{1-\mu_{0_+}}{2}\bigr)}{x^2} \nonumber\\ 
  &\hspace{4.7cm}+m_0\frac{(\epsilon_1\!\!+\!\epsilon_2)\mu_{0_+}\!\!-\epsilon_2(\nu _{0_+}\!\!-\nu _{0_-})}{x^2} -\frac{x+4m_0^2}{4  x^2}\Bigg)\psi^{\rm TR} (x) =0
\end{align}
In the NS limit, this can be written
\begin{align}
  \Bigg(\epsilon_1 ^2 \dfrac{d^2}{dx^2}  -\frac{x+4{\rm m}_{0}^{+}{\rm m}_{0}^{-}}{4  x^2}\Bigg)\psi^{\rm NS} (x) =0.
\end{align}
where ${\rm m}_{0}^{\pm}=m_0-\frac{\epsilon_1\mu_{0_+} }{2}\pm \frac{\epsilon_1}{2}$.



\subsubsection*{Whittaker:}

The functions defining the spectral curve are given by
\begin{align}
x(z)=    2 m_{\infty } \left(\frac{1}{z-1}-\frac{1}{z+1}\right) ,\quad y(z)= \frac{1}{2}z
\end{align}
with the involution $\sigma(z)=-z$. The (effective) ramification points are $\mathcal{R}=\mathcal{R}^*=\{0,\infty\}$, and we have $\mathcal{P}=\{-1,1\}$, $\widetilde{\mathcal{P}}=\widetilde{\mathcal{P}}_+=\varnothing$.\medskip

The quantum curve is identical to the Kummer case \eqref{eq:qckummer} with the parameters $m_0$, $\nu_{0_\pm}$ set to zero, and $\mu_{0_+}=\pm1$.


\subsubsection*{Weber:}


The functions defining the spectral curve are given by
\begin{align}
    x(z)= \sqrt{m_\infty}\left(z+\dfrac{1}{z}\right),\quad y(z)=\dfrac{\sqrt{m_\infty}}{2}\left(z-\dfrac{1}{z}\right)
\end{align}
with the involution $\sigma(z)=1/z$. The (effective) ramification points at $\mathcal{R}=\mathcal{R}^*=\{+1,-1\}$, and we have ${\mathcal{P}}=\widetilde{\mathcal{P}}=\{0,\infty\}$, with $\mathcal{\widetilde{P}}_+=\{\infty\}$. \medskip

\noindent The quantum curve is given explicitly by:
\begin{equation}
\left(
\epsilon_1^2 \dfrac{d^2}{dx^2} + \frac{ \epsilon_2(\nu_{\infty_+}\!\! - \nu_{\infty_-})-(\epsilon_1\!\!+\!\epsilon_2)\mu_{\infty_+}}{2} + \frac{4 m_\infty -x^2}{4  }  \right) \psi^{\rm TR} (x) =0
\end{equation}
In the NS limit, this can be written
\begin{equation}
\left(
\epsilon_1^2 \dfrac{d^2}{dx^2} + \frac{2 ({\rm m}_\infty^++{\rm m}_\infty^-) -x^2}{4  }  \right) \psi^{\rm NS} (x) =0.
\end{equation}
where ${\rm m}_{\infty}^{\pm}=m_\infty-\frac{\epsilon_1\mu_{\infty_+} }{2}\pm \frac{\epsilon_1}{2}$.

\subsubsection*{Degenerate Bessel:}
The functions defining the spectral curve are given by
\begin{align}
    x(z)=z^2,\quad y(z)=\frac{1}{2z}
\end{align}
with the involution $\sigma(z)=-z$. The (effective) ramification points are $\mathcal{R}=\{0,\infty\}$, $\mathcal{R}^*=\{0\}$, and we have $\mathcal{P}=\widetilde{\mathcal{P}}=\widetilde{\mathcal{P}}_+=\varnothing$. \medskip

\noindent The quantum curve is identical to the Bessel case \eqref{eq:qcbessel} with the parameters $m_0$ $\nu_{0_{\pm}}$ set to zero, and $\mu_{0_+}=\pm1$.


\subsubsection*{Airy:}
The functions defining the spectral curve are given by
\begin{align}
    x(z)=z^2,\quad y(z)=z
\end{align}
with the involution $\sigma(z)=-z$. The (effective) ramification points are $\mathcal{R}=\{0,\infty\}$, $\mathcal{R}^*=\{0\}$ and we have $\mathcal{P}=\widetilde{\mathcal{P}}=\widetilde{\mathcal{P}}_+=\varnothing$.\medskip

\noindent The quantum curve is explicitly given by:
\begin{equation}
    \left(\epsilon_1^2\frac{d}{dx^2}-x\right)\psi^{\rm TR}(x)=0.
\end{equation}

\appendix
\section{Proofs of lemmas}\label{sec:proofs}
In this appendix, we give the proofs of several lemmas appearing in Section~\ref{sec:properties}.

\begin{lemma}[Lemma \ref{lem:pole1/2,1}]
All poles of
$\omega_{\frac12,1}$ are simple and lie in $\mathcal{R}\cup\widetilde{\mathcal{P}}$. Explicitly,
let ${\rm ord}_p(y)$ be the order of the zero of $y$ at $p\in\widetilde{\mathcal{P}}$ and $-{\rm ord}_p(y)$ be the order of the pole, then $\omega_{\frac12,1}$ is given by
\begin{equation}
    \omega_{\frac12,1}(p_0)=\frac{\mathscr{Q}}{2}\left(-\frac{dy(p_0)}{y(p_0)}+\sum_{p\in \widetilde{\mathcal{P}}_+}\,\mu_p\,{\rm ord}_p(y)\,\eta_p(p_0)\right),\label{w1/2,1-2A}
\end{equation}
where $\eta_p$ is the unique meromorphic differential on $\mathbb{P}^1$ with residue $+1$ at $p$ and $-1$ at $\sigma(p)$.
\end{lemma}

\begin{proof}
\hfill

\noindent One can see from the definition of the recursion kernel \eqref{kernel} that the definition of $\omega_{\frac12,1}$ \eqref{eq:half1def} is simplified as
\begin{equation}
  \omega_{\frac12,1}(p_0)=-\frac{\mathscr{Q}}{2}\left(\sum_{r\in\mathcal{R}}\underset{p=r}{\text{Res}}+\underset{p=\sigma(p_0)}{\text{Res}}+\sum_{r\in\widetilde{\mathcal{P}}_+}\mu_r\cdot\underset{p=r}{ \text{Res}}\,\right)\frac{dy(p)}{y(p)}\int_{\sigma(p)}^p\omega_{0,2}(p_0,\cdot),\label{w1/2,1-0}
\end{equation}
where we used $y(\sigma(p))=-y(p)$. Therefore, the integrand can only have poles with respect to $p$ at $p=p_0$ and $p=\sigma(p_0)$ or at the zeroes and poles of $y$ (that is, $\mathcal{R}$ and $\widetilde{\mathcal{P}}$).  Since the sum of all residues on $\Sigma$ vanishes,
\begin{equation}
    \omega_{\frac12,1}(p_0)=\frac{\mathscr{Q}}{2}\left(-\frac{dy(p_0)}{y(p_0)}-\sum_{r\in \widetilde{\mathcal{P}}_-}\mu_r\cdot\underset{p=r}{\text{Res}}\,\frac{dy(p)}{y(p)}\int^p_{\sigma(p)}\omega_{0,2}(p_0,\cdot)\right).\label{w1/2,1-1}
\end{equation}
Using the standard fact that 
\begin{equation}
\int_{\sigma(p)}^p\omega_{0,2}(p_0,\cdot)=\eta_p(p_0)
\end{equation}
we evaluate the remaining residue, and using the anti-invariance under the involution of the second term, we arrive at \eqref{w1/2,1-2A}. 
The pole structure immediately follows.
\end{proof}

\begin{lemma}[Lemma \ref{lem:poles}]
For $2g+n\geq2$, all poles of $\omega_{g,n+1}(p_0,J)$, with respect to $p_0$, lie in $\mathcal{R}^*\cup\sigma(J)$.
\end{lemma}
\begin{proof}
\hfill

\noindent We proceed by induction on $2g+n$. For $2g+n>2$, we will directly evaluate the recursion formula \eqref{eq:RTR}. It turns out that the first case $2g+n=2$ is considerably more involved.

 When $(g,n)=(0,2)$, Remark~\ref{rem:0n} implies that the pole structure of $\omega_{0,3}$ obeys TR\ref{TR2}, as desired. For $(g,n)=(1,0)$, we recall that $\text{Rec}_{1,1}^{\mathscr Q}(p_0)$ is given by
\begin{align}
   \text{Rec}^\mathscr{Q}_{1,1}(p_0)=&\;\;\omega_{\frac12,1}(p_0)\cdot\omega_{\frac12,1}(p_0)-\omega_{0,2}(p_0,\sigma(p_0))+\mathscr{Q}\, dx(p_0)\cdot d_{p_0}\left(\frac{\omega_{\frac12,1}(p_0)}{dx(p_0)}\right).
\end{align}
The pole structure of $\omega_{\frac{1}{2},1}$ (Lemma~\ref{lem:pole1/2,1}) shows that poles of $\text{Rec}_{1,1}(p_0)$ lying in $\widetilde{\mathcal{P}}$ can be at worst of second order. Let ${\mathcal{P}}^{(1)}\subset{\mathcal{P}}$ denote the set of simple poles of $ydx$. Then since the kernel has $\Delta ydx$ in the denominator, ${\rm K}(p_0,p)\text{Rec}^\mathscr{Q}_{1,1}(p)$ does not have poles (in $p$) at $\widetilde{\mathcal{P}}$ unless $p\in\mathcal{P}^{(1)}$. As a result, the recursion formula \eqref{eq:RTR} gives:
\begin{equation}
    \omega_{1,1}(p_0)=-\frac{\text{Rec}_{1,1}^Q(p_0)}{2\omega_{0,1}(p_0)}+2\!\!\!\!\sum_{q\in\mathcal{P}^{(1)} \cap\, \widetilde{\mathcal{P}}_-}\!\!\!\!\underset{p=q}{\text{Res}}\,{\rm K}(p_0,p)\left(\left(\omega_{\frac12,1}(p)^2\right)+\mathscr{Q}dx(p)\cdot d_{p}\frac{\omega_{\frac12,1}(p)}{dx(p)}\right),\label{w1,1-1}
\end{equation}
where we used again that the sum of the residues is zero. From the formula \eqref{w1,1-1}, it follows that the only poles of $\omega_{1,1}$ must lie in $\mathcal{R}^*$ or $\widetilde{\mathcal{P}}$.

The remaining task is to show that $\omega_{1,1}$ has no poles at $\widetilde{\mathcal{P}}$. Let us focus on a pair $p_{\pm}\in\mathcal{P}^{(1)} \cap\, \widetilde{\mathcal{P}}_{\pm}$ where $\sigma(p_-)=p_+$.
Plugging the explicit form of $\omega_{0,1}$ and $\omega_{\frac12,1}$ into \eqref{w1,1-1} and expanding in local coordinates near $p_{\pm}$ allows us to check that, term by term, any pole appearing at $\widetilde{\mathcal{P}}$ can be at worst simple. We check that it in fact has no residue: the first term gives
\begin{align}
    \underset{p=p_{\pm}}{\text{Res}}&\left(-\frac{\text{Rec}^{\mathscr{Q}}_{1,1}(p)}{2\omega_{0,1}(p)}\right)=-\frac{\mathscr{Q}^2}{8 \underset{{q=p_{\pm}}}{\text{Res}}\omega_{0,1}(q)}\left(\mu_{p_{\pm}}^2-1\right).
\end{align}
and a direct computation of the second term cancels out exactly with this, so that $\omega_{1,1}$ has no poles at $\widetilde{\mathcal{P}}$.

For $(g,n)=(\frac12,1)$, the pole structure of $\omega_{\frac12,1}$ ensures that the integrand ${\rm K}(p_0,p)\text{Rec}^{\mathscr{Q}}_{\frac12,2}(p,p_1)$ can have poles in $p$ only at $\mathcal{R}^*$, $J_0$, and $\sigma(J_0)$ (but not at $\widetilde{\mathcal{P}}$), so the recursion formula \eqref{eq:RTR} gives:
\begin{equation}
   \omega_{\frac12,2}(p_0,p_1)=-\frac{\text{Rec}^{\mathscr Q}_{\frac12,2}(p_0,p_1)}{2\omega_{0,1}(p_0)} +2d_{p_1}\left({\rm K}(p_0,p_1)\cdot \omega_{\frac12,1}(p_1)\right),\label{P1/2,20}
\end{equation}
where we again used the fact that the sum of residues is zero. 

The pole structure of  $\omega_{\frac12,1}$ implies, together with \eqref{P1/2,20}, that the only poles of $\omega_{\frac12,2}$ lie in $\mathcal{R}^*$, $J_0$, or $\sigma(J_0)$, hence we only have to show that it is regular (in $p_0$) at $p_1$. This can be shown by a direct expansion of \eqref{P1/2,20} in local coordinates around $p_1$. In particular, the poles coming from $2\omega_{0,2}$ in ${\rm Rec}^{\mathscr{Q}}_{\frac12,2}$ cancel out with the  poles coming from the second term in \ref{P1/2,20}.


We now assume Lemma~\ref{lem:poles} up to $m=2g+n\geq2$. Then it implies that for any $g,n$ with $2g+n=m+1$, the integrand in the recursion formula \eqref{eq:RTR} has no poles at $\widetilde{\mathcal{P}}$ because $\omega_{\frac12,1}$ only appears linearly and $ydx$ has a pole there. We can equivalently take residues at $J_0$ (since the sum of residues vanishes) in the recursion \eqref{eq:RTR}. Then the only possible contributions of residues at $p=p_i\in J$ are from $\omega_{0,2}(p,p_i)$. Thus, it follows:
\begin{equation}
    \omega_{g,n+1}(p_0,J)=-\frac{\text{Rec}^{\mathscr{Q}}_{g,n+1}(p_0,J)}{2\omega_{0,1}(p_0)} +2\sum_{i=1}^n d_{p_i}\left({\rm K}(p_0,p_i)\cdot\omega_{g,n}(J)\right).\label{pole2}
\end{equation}
It is clear that there is no pole at $\widetilde{\mathcal{P}}$. A very similar calculation in local coordinates around $p_i \in J$ to the case of $\omega_{\frac12,2}$ shows that the poles in $p_0$ from the first term cancel with the sum term. 

Thus, the only possible poles of $\omega_{g,n+1}(p_0,J)$, with respect to $p_0$, are at $\mathcal{R}^*$ and at $\sigma(J)$. 
\end{proof}

\begin{lemma}[Lemma \ref{lem:GLP}]Define a quadratic (in $p$) differential $P_{g,n+1}(p,J)$ for $2g+n\geq2$ by
\begin{equation}
    P_{g,n+1}(p,J):=2\omega_{0,1}(p)\omega_{g,n+1}(p,J)+{\rm Rec}_{g,n+1}^\mathscr{Q}(p,J).
\end{equation}
Then:
\begin{enumerate}[label=\roman*)]
    \item $P_{g,n+1}$ is invariant under the involution $\sigma$,
    \item $P_{g,n+1}$ is holomorphic at every effective ramification point $r\in\mathcal{R}^*$.
\end{enumerate}
\end{lemma}

\begin{proof}
\hfill

\noindent The integrand in the recursion \eqref{eq:RTR} for $\omega_{0,3}(p_0,p_1,p_2)$ can have poles at $\mathcal{R}$, $J_0$, and $\sigma(J_0)$. Since sum of all residues vanishes, we can write it as
 
\begin{align}\omega_{0,3}(p_0,p_1,p_2)=\,&2\sum_{r\in J_0}\underset{p=r}{\text{Res}}\,{\rm K}(p_0,p)\!\cdot\text{Rec}_{0,3}^{\mathscr{Q}}(p,J),\\
   =&-\frac{\text{Rec}^{\mathscr{Q}}_{0,3}(p_0,p_1,p_2)}{2\omega_{0,1}(p_0)} +\sum_{\substack{i,j=1\\i\neq j}}^2 d_{p_i}\left({\rm K}(p_0,p_i)\cdot\Delta\omega_{0,2}(p_i,p_j)\right),
\end{align}
where the first term comes from taking the residue at the pole $p=p_0$ of ${\rm K}(p_0,p)$ and the second term is the contributions from the double pole of $\omega_{0,2}(p,p_i)$. This means that 
\begin{equation}
    P_{0,3}(p_0,p_1,p_2)=2\omega_{0,1}(p_0)\sum_{\substack{i,j=1\\i\neq j}}^2d_{p_i}\Bigl({\rm K}(p_0,p_i)\cdot\Delta\omega_{0,2}(p_i,p_j)\Bigr).\label{P03}
\end{equation}
Notice that $P_{0,3}$ is invariant under the involution in $p_0$. Furthermore, ${\rm K}(p_0,p_i)$ on the right hand side of \eqref{P03} has no poles at $p_0=r\in\mathcal{R}^*$ and $\omega_{0,1}$ is non-singular at each effective ramification point by definition (in fact due to the anti-invariance under the involution and Assumption \ref{ass:tr}, the order of any zero of $\omega_{0,1}$ is always two). Thus, the loop equations hold for $\omega_{0,3}$.

From \eqref{w1,1-1} in the previous proof, it follows
\begin{equation}
    P_{1,1}(p_0)=2\omega_{0,1}(p_0)\sum_{q\in\mathcal{P}^{(1)} \cap\, \widetilde{\mathcal{P}}_-}\underset{p=q}{\text{Res}}\,{\rm K}(p_0,p)\left(\omega_{\frac12,1}(p)\cdot\omega_{\frac12,1}(p)+\mathscr{Q}\,dx(p)\cdot d_{p}\left(\frac{\omega_{\frac12,1}(p)}{dx(p)}\right)\right).\label{P1,1-1}
\end{equation}
By similar reasoning to the case of $P_{0,3}$, we conclude this satisfies the loop equations.


Finally, from \eqref{P1/2,20}, \eqref{pole2}, we find that for all other cases
\begin{equation}
    P_{g,n+1}(p_0,J)=4\omega_{0,1}(p_0)\sum_{i=1}^n d_{p_i}\Bigl({\rm K}(p_0,p_i)\cdot \omega_{g,n}(J)\Bigr).\label{Pg,n+1-1}
\end{equation}
This satisfies the loop equations by a similar reasoning as above.
\end{proof}

\begin{lemma}[Lemma \ref{lem:w12,2}]The bidifferential $ \omega_{\frac12,2}$ is symmetric:
\vspace{-1mm}
\begin{equation}
    \omega_{\frac12,2}(p_0,p_1)=\omega_{\frac12,2}(p_1,p_0).
\end{equation}
\end{lemma}
\begin{proof}
\hfill

\noindent The recursion for $\omega_{\frac12,2}$ reads
\allowdisplaybreaks[0]
\begin{align}
    \omega_{\frac12,2}(p_0,p_1)=-2&\left(\sum_{r\in\mathcal{R}}\underset{p=r}{\text{Res}}+\underset{p=\sigma(p_0)}{\text{Res}}+\underset{p=\sigma(p_1)}{\text{Res}}\right){{\rm K}(p_0,p)}\nonumber\\
    &\hspace{1cm}\times \biggl(\omega_{\frac12,1}(p)\cdot\Delta\omega_{0,2}(p,p_1)-\mathscr{Q}\,dx(p)\cdot d_p\left(\frac{\omega_{0,2}(\sigma(p),p_1)}{dx(p)}\right)\biggr).\label{w1/2,21}
\end{align}
\allowdisplaybreaks[1]
We will focus on the second term inside the parentheses in \eqref{w1/2,21}. For any curve satisfying our assumptions, \eqref{w02base} holds which is equivalently written as
\begin{equation}
    -\omega_{0,2}(\sigma(p),p_1)=-\frac{dx(p)dx(p_1)}{2(x(p)-x(p_1))^2}+\frac12\Delta\omega_{0,2}(p,p_1).
\end{equation}
Note that
\begin{equation}
    -2\left(\sum_{r\in\mathcal{R}}\underset{p=r}{\text{Res}}+\underset{p=\sigma(p_0)}{\text{Res}}+\underset{p=\sigma(p_1)}{\text{Res}}\right){{\rm K}(p_0,p)}\cdot dx(p)\cdot d_p\left(\frac{dx(p_1)}{2(x(p)-x(p_1))^2}\right)=0.
\end{equation}
This follows since the integrand is invariant under the involution $\sigma$, so by Lemma~\ref{lem:contour} we may compute it as a sum of residues over $\mathcal{R}$. Since the residue of a total derivative of a meromorphic function vanishes (equivalently, integrating by parts), the second term of \eqref{w1/2,21} becomes
\begin{align}
    \mathscr{Q}\left(\sum_{r\in\mathcal{R}}\underset{p=r}{\text{Res}}+\underset{p=\sigma(p_0)}{\text{Res}}+\underset{p=\sigma(p_1)}{\text{Res}}\right)\left(\frac{\Delta\omega_{0,2}(p,p_0)}{4\, y(p)}-{\rm K}(p_0,p)\cdot dx(p)\cdot  \frac{dy(p)}{y(p)}\right)\frac{\Delta\omega_{0,2}(p,p_1)}{dx(p)}.
\end{align}
Together with \eqref{w1/2,1-2A}, we find
\begin{align}
    \omega_{\frac12,2}(p_0,p_1)={\mathscr{Q}}\left(\sum_{r\in\mathcal{R}}\underset{p=r}{\text{Res}}+\underset{p=\sigma(p_0)}{\text{Res}}+\underset{p=\sigma(p_1)}{\text{Res}}\right)&\Biggl(\frac{\Delta\omega_{0,2}(p,p_0)\cdot\Delta\omega_{0,2}(p,p_1)}{4\, y(p)dx(p)}\nonumber\\
    &-{\rm K}(p_0,p)\cdot\Delta\omega_{0,2}(p,p_1)\cdot\!\!\!\sum_{q\in \widetilde{\mathcal{P}}_+}\,\mu_q\,{\rm ord}_q(y)\eta_q(p)\Biggr).\label{w1/2,22}
\end{align}
The first line of \eqref{w1/2,22} is manifestly symmetric under $p_0\leftrightarrow p_1$, hence we are left with proving the symmetry for the second line. Note that the integrand of the second line is invariant under the involution, hence Lemma~\ref{lem:contour} shows that it can be written as 
\begin{equation}
  -\frac{\mathscr{Q}}{2}\sum_{q\in \widetilde{\mathcal{P}}_+} \sum_{r\in\mathcal{R}}\underset{p=r}{\text{Res}}\, {\rm K}(p_0,p)\biggl(\mu_q\,{\rm ord}_q(y)\eta_q(p)\cdot\Delta\omega_{0,2}(p,p_1)\Biggr).\label{w1/2,23}
\end{equation}
\noindent Thus since the residue is taken only at ramification points, after evaluation the only possible poles in $p_0$ and $p_1$ must lie in $\mathcal{R}$.

We will show that the contribution of \eqref{w1/2,23} at each ramification point is symmetric. As explained in \cite{ABCD,BBCCN,BO2,O2}, in a neighbourhood of $r\in\mathcal{R}$, $\omega_{0,2}$ is expanded in a local coordinate $z:=z(p)$ so that
\begin{align}
   \Delta\omega_{0,2}(p,p_1)&=\sum_{k>0}(2k-1)d\xi_{-2k+1}(p_1)\cdot z^{2k-2}dz,\\
   \int_{\sigma(p)}^p\omega_{0,2}(p_0,\cdot) &=\sum_{k>0}d\xi_{-2k+1}(p_0)\cdot z^{2k-1}(p),
\end{align}
where $d\xi_{-2k+1}$ is a meromorphic differential locally defined around $r\in\mathcal{R}$, and whose only pole is of order $2k$ of coefficient 1. See \cite{ABCD,BBCCN,BO2,O2} for the details. Since $\omega_{0,1}$ has at most a double zero at $\mathcal{R}$ by Assumption \ref{ass:tr}, the residue computation in \eqref{w1/2,23} at each $r\in\mathcal{R}$ gives a contribution of
\begin{equation}
    c_{r}\,d\xi_{-1}(p_0)d\xi_{-1}(p_1),\label{w1/2,24}
\end{equation}

\noindent where $c_{r}$ is a coefficient determined by the residue which necessarily vanishes unless $\omega_{0,1}$ has a double zero at $r$. The result follows, since \eqref{w1/2,24} is clearly symmetric under the exchange $p_0\leftrightarrow p_1$.
\end{proof}

\begin{lemma}[Lemma \ref{lem:residue3}]
Given a genus zero degree two refined spectral curve, let $p_i \in\Sigma$ for $i=0,\ldots n$ and suppose that $p_i\not\in\mathcal{P}$, and $p_i\neq \sigma(p_j)$ for any $i,j$. Suppose furthermore that lower endpoints $q_i\in {\mathcal{P}}$ are chosen. Then for all $2g+n\geq2$, the integral
\begin{equation}
   \int^{p_0}_{q_0}\cdots\int^{p_{n}}_{q_{n}}\omega_{g,n+1}(\zeta_0,\zeta_1,\ldots,\zeta_n)\label{residue3A}
\end{equation}
where the $(i+1)th$ integral from the left is in the $\zeta_i$ variable, exists and defines a meromorphic function on $\Sigma$ with respect to each $p_i$, regular at $p_i=p_j$ for any $i,j$, and is independent of the path of integration.

\end{lemma}
\begin{proof}

For convenience, in the proof we reverse the order of integrals, so that we integrate in $\zeta_0$ last rather than first. We will actually prove the result that for all $2g+n\geq2$, the intermediate multi-integral of $\omega_{g,n+1}$ with respect to the last $k$ variables
\begin{equation}
    I_{g,n+1}^{(k)}(\zeta_0,,...,\zeta_{n-k};p_{n-k+1},...,p_n):=
    \int_{q_{n-k+1}}^{p_{n-k+1}}\cdots\int_{q_n}^{p_n}\omega_{g,n+1}(\zeta_0,...,\zeta_n),
\end{equation}
exists whenever all $\zeta_i\neq\sigma(\zeta_j)$ and $\zeta_i \neq \sigma(p_j)$, and defines a meromorphic function on $\Sigma$ in $p_i$ (regular at $p_j$), and a residue-free meromorphic differential with respect to each $\zeta_i$. The result then follows when $k=n+1$.

Towards induction, for $2g+n=2$ we first note that Proposition~\ref{prop:Q=0RTR} implies the result for $I_{0,3}^{(k)}$, $k=0,\ldots,3$ and for $I_{1,1}^{(k)}$, $k=0,1$ follows from Lemma~\ref{lem:poles}. 

Since the only poles that could cause a problem integrating $\omega_{\frac12,2}$ are at $\zeta_1=\sigma(\zeta_0)$, by the assumption on the endpoints (the path can be taken avoiding any others) $I^{(1)}_{\frac12,2}$ exists, and is independent of the path thanks to RTR\ref{RTR3}. 
Due to the pole structure of $\omega_{\frac12,2}$ it is clear the only new poles (in $\zeta_0$) appearing in $I^{(1)}_{\frac12,2}$ could 
be at $\zeta_0=\sigma(p_1)$ or $\zeta_0=\sigma(q_1)$. To see if $I^{(2)}_{\frac12,2}$ exists, we must check whether $I^{(1)}_{\frac12,1}$ has a pole in $\zeta_0$ at the endpoints $p_0$, $q_0$. The former is ruled out by the assumption that $p_0\neq\sigma(p_1)$. To check the latter, we first evaluate the residues in \eqref{w1/2,22}, \eqref{w1/2,23} as in the previous proof, to write
\begin{equation}
    \omega_{\frac12,2}(\zeta_0,\zeta_1)= -\mathscr{Q}\left(d_{\zeta_0}\frac{\Delta\omega_{0,2}(\zeta_0,\zeta_1)}{4\, y(\zeta_0)dx(\zeta_0)}+d_{\zeta_1}\frac{\Delta\omega_{0,2}(\zeta_1,\zeta_0)}{4\, y(\zeta_1)dx(\zeta_1)}\right)+\text{reg},\label{eq:w1/2,25}
\end{equation}
where reg denotes terms regular as $\zeta_0\to\sigma(\zeta_1)$ (the contribution from the second line in \eqref{w1/2,22}).

\noindent 
Explicitly integrating the first two terms in \eqref{eq:w1/2,25} from $q_1$ to $p_1$, we obtain
\begin{align}
 &\int_{q_1}^{p_1} -\mathscr{Q}\left(d_{\zeta_0}\frac{\Delta\omega_{0,2}(\zeta_0,\zeta_1)}{4\, y(\zeta_0)dx(\zeta_0)}+d_{\zeta_1}\frac{\Delta\omega_{0,2}(\zeta_1,\zeta_0)}{4\, y(\zeta_1)dx(\zeta_1)}\right)\nonumber\\
    &= -\mathscr{Q}\left(d_{\zeta_0}\frac{\eta^{p_1}(\zeta_0)-\eta^{q_1}(\zeta_0)}{4\, y(\zeta_0)dx(\zeta_0)}+\frac{\Delta\omega_{0,2}(z_1,\zeta_0)}{4\, y(p_1)dx(p_1)}-\frac{\Delta\omega_{0,2}(q_1,\zeta_0)}{4\, y(q_1)dx(q_1)}\right),\label{eq:w1/2,26}
\end{align}

\noindent At $\zeta_0=\sigma(q_1)$ the last term vanishes since $\sigma(q_1)\in{\mathcal{P}}$, and since the denominator $\omega_{0,1}(\zeta_0)$ in the first term has a pole there, the only potentially singular term $\eta^{q_1}(\zeta_0)$ is in fact regular at $\zeta_0= \sigma(q_1)$. 
 
 Thus, irrespective of the choice of $q_0\in\mathcal{P}$, $I^{(1)}_{\frac12,2}$ is holomorphic in $\zeta_0$ along the (next) path, so that $I^{(2)}_{\frac12,2}$ exists. Furthermore, the result is independent of the path since $I^{(1)}_{\frac12,2}$ is residue-free in $\zeta_0$ (this follows from \ref{eq:w1/2,26} and the product structure of the ``reg" terms,
 described in \ref{w1/2,24}). Since $I^{(1)
 }_{\frac12,2}$ was regular at $\zeta_0=p_1$,
 $I^{(2)}_{\frac12,2}(p_0,p_1)$ is regular at $p_0=p_1$.

We now proceed by induction on both $2g+n$ and $k$. 
First we assume that for all $(\tilde{g},\tilde{n})$ with $2\tilde{g}+\tilde{n} < 2g+n $ and for all $0\leq \tilde{k} \leq \tilde{n}+1$  the integral $I_{\tilde{g},\tilde{n}+1}^{(\tilde{k})}$
\begin{equation}
    \int_{q_{\tilde{n}-\tilde{k}+1}}^{p_{{\tilde{n}}-\tilde{k}+1}}\cdots\int_{q_{\tilde{n}}}^{p_{\tilde{n}}}\omega_{\tilde{g},\tilde{n}+1}(\zeta_0,...,\zeta_{\tilde{n}})
\end{equation}
exists and defines a residue-free meromorphic differential in $\zeta_0,...,\zeta_{\tilde{n}-\tilde{k}}$ and a meromorphic function in $p_{\tilde{n}-\tilde{k}+1},...,p_{\tilde{n}}$, which is regular at $\zeta_i= \sigma(q_j)$ for $i\in\{1,\ldots,\tilde{n}-\tilde{k}\}$ and $j\in\{\tilde{n}-\tilde{k}+1,...,\tilde{n}\}$. 

Now we consider $(g,n)$ and induct on $k$. The result holds trivially for $I^{(0)}_{g,n}=\omega_{g,n+1}$, and we assume inductively it holds for $I^{(k)}_{g,n}$. Then the integral $I_{g,n+1}^{(k+1)}$ in $\zeta_{n-k}$ exists by assumption. When $k=n$, we are done integrating, and by the residue-freeness $I^{(n+1)}_{g,n+1}$ thus defines a meromorphic function in $p_0,...,p_n$. In the case $0\leq k <n$ (so there are at least two more integrals to do), since $I^{(k+1)}_{g,n+1}$ is symmetric in $\zeta_0,...,\zeta_{n-k}$, it is sufficient to show it has no poles at $\zeta_0=\sigma(q_j)$ and it is residue-free with respect to $\zeta_0$.

We proceed by integrating \eqref{pole2}. Although $I_{\tilde{g},\tilde{n}+1}^{(\tilde{k})}$ is regular at $\zeta_0=\sigma(q_j)$ for $j>\tilde{n}-\tilde{k}$ by the inductive assumption, if it contains $\zeta_{n-k}$ as an argument it is \emph{singular} at $\zeta_0=\sigma(\zeta_{n-k})$, and thus one has to confirm that $I_{g,n+1}^{(k+1)}$ is regular at $\zeta_0=\sigma(q_{n-k})$ after integrating with respect to $\zeta_{n-k}$. Let us focus on the contribution from the second term in \eqref{pole2}. Since the $\zeta_0$-dependence only appears in $\mathrm{K}(\zeta_0,\zeta_i)$, the potentially problematic term is the following total derivative:
\begin{equation}
    2d_{\zeta_{n-k}}\left(\mathrm{K}(\zeta_0,\zeta_{n-k})\cdot I_{g,n}^{(k)}(\zeta_1,...,\zeta_{n-k},p_{n-k+1},...,p_{n})\right).
\end{equation}
However, since the denominator of $\mathrm{K}(\zeta_0,\zeta_{n-k})$ is $\omega_{0,1}(\zeta_{n-k})$, and since $I_{g,n}^{(k)}$ is regular as $\zeta_{n-k}=\sigma(q_i)$, the contribution from the endpoint $q_{n-k}$ after integration with respect to $\zeta_{n-k}$ vanishes. Thus, the second term in \eqref{pole2} becomes regular at $\zeta_0=\sigma(q_{n-k})$ after integrating with respect to $\zeta_{n-k}$.

We next consider the contribution from the first term in \eqref{pole2}. By the inductive assumption, most terms (containing $I^{(k')}$, $k'<k$) trivially become regular at $\zeta_0=\sigma(q_{n-k})$. However, since the inductive assumption does not apply to $\omega_{0,2}$, potentially problematic terms are:
\begin{equation}
    -\frac{1}{2\omega_{0,1}(\zeta_0)}\cdot\Delta\omega_{0,2}(\zeta_0,\zeta_{n-k})\cdot I_{g,n}^{(k)}(\zeta_0,...,\zeta_{n-k-1},p_{n-k+1},...,p_{n}).\label{A51}
\end{equation}
Although
\begin{equation}
    \int_{q_{n-k}}^{p_{n-k}}\Delta\omega_{0,2}(\zeta_0,\zeta_{n-k})
\end{equation}
becomes singular at $\zeta_0=\sigma(q_{n-k})$, \eqref{A51} itself becomes regular there due to the presence of $\omega_{0,1}(\zeta_0)$ in the denominator. Thus $I_{g,n+1}^{(k+1)}$ is regular at $\zeta_0=\sigma(q_{n-k})$.

It remains to show that $I_{g,n+1}^{(k+1)}$ is residue-free with respect to $\zeta_0$. Note that it must have vanishing residue at $\mathcal{R}$ because $\omega_{g,n+1}$ itself does, and the pole order at $\mathcal{R}$ does not change when one integrates with respect to other variables. The inductive assumption ensures that there is no residue at $\zeta_0=\sigma(p_j)$ whenever $j\neq n-k$, so that the only remaining $\zeta_0\in\Sigma$ with nonzero residue would be at $\zeta_0=\sigma(p_{n-k})$. Since the sum of all residues vanishes, this one must too.

\end{proof}

\bibliography{refs}
\bibliographystyle{utphys.bst}

\end{document}